\documentclass[a4paper, 11pt]{article}
\usepackage{amsmath}
\usepackage{amssymb,esint}
\usepackage{amscd}
\usepackage{xspace}
\usepackage{fancyhdr}
\usepackage{color}
\usepackage{srcltx} 

\DeclareMathOperator*{\argmin}{arg\,min}

\newtheorem{theorem}{Theorem}[section]

\newtheorem{corollary}[theorem]{Corollary}

\newtheorem{lemma}[theorem]{Lemma}

\newtheorem{remark}[theorem]{Remark}

\newenvironment{proof}[1][Proof]{\textbf{#1.} }{\hfill\rule{0.5em}{0.5em}}
{\catcode`\@=11\global\let\AddToReset=\@addtoreset
	\AddToReset{equation}{section}
	
	\AddToReset{theorem}{section}

	\def\nc{\newcommand}
	\def\al{\alpha}\def\de{\delta}
	
	\def\ep{\epsilon}\def\ka{\kappa}
	
	\def\om{\omega}

	 \def\Om{\Omega}

	\nc\pa{\partial}
	
	\nc\CC{\mathbb{C}}
	\nc\RR{\mathbb{R}}
	\nc\QQ{\mathbb{Q}}
	\nc\ZZ{\mathbb{Z}}
	\nc\NN{\mathbb{N}}

	\newcommand{\vertiii}[1]{{\left\vert\kern-0.25ex\left\vert\kern-0.25ex\left\vert #1 
			\right\vert\kern-0.25ex\right\vert\kern-0.25ex\right\vert}}

\begin{document}	
		\title{Universal potential estimates for $1<p\leq 2-\frac{1}{n}$}
		\author{
			{\bf Quoc-Hung Nguyen\thanks{E-mail address: : qhnguyen@amss.ac.cn, Academy of Mathematics and Systems Science, Chinese Academy
					of Sciences, Beijing, 100190, China.} ~ and~ Nguyen Cong Phuc\thanks{E-mail address: pcnguyen@math.lsu.edu, Department of Mathematics, Louisiana State University, 303 Lockett Hall,
					Baton Rouge, LA 70803, USA. }} \vspace{0.3in}\\
				{\it Dedicated to Giuseppe Mingione on the occasion of his 50th birthday}
			}
		\date{}  
		\maketitle
		\begin{abstract}
		 We extend the so-called universal potential estimates of Kuusi-Mingione type (J.   Funct. Anal. 262:  4205--4269, 2012) to the singular case $1<p\leq 2-1/n$ for the quasilinear equation 
		 with measure data
		 \begin{equation*}
		 -\operatorname{div}(A(x,\nabla u))=\mu
		 \end{equation*}
		 in a bounded open subset $\Omega$ of $\mathbb{R}^n$, $n\geq 2$,  with a finite signed measure $\mu$ in $\Omega$. The operator $\operatorname{div}(A(x,\nabla u))$
		 is modeled after the $p$-Laplacian   $\Delta_p u:= {\rm div}\, (|\nabla u|^{p-2}\nabla u)$, where the nonlinearity $A(x, \xi)$ ($x, \xi \in \RR^n$) is assumed to satisfy
		 natural growth and monotonicity conditions of order $p$, as well as certain additional regularity conditions in the $x$-variable.
		
			\medskip
			
			\medskip
			
			\medskip
			
			
			\medskip
			
		\end{abstract}   
		
		\tableofcontents
		
		\section{Introduction and main results}

	We are concerned here with  the quasilinear elliptic equation with measure data 
	\begin{equation}\label{quasi-measure}
	-\operatorname{div}(A(x,\nabla u))=\mu,
	\end{equation}
	 in a bounded open subset $\Omega$ of $\mathbb{R}^n$, $n\geq 2$. Here $\mu$ is a finite signed measure in $\Omega$ and
	the nonlinearity  $A = (A_1, \dots, A_n):\mathbb{R}^n\times \mathbb{R}^n\to \mathbb{R}^n$ is vector valued function.
	Throughout the paper, we assume that there exist $\Lambda\geq 1$ and $p>1$ such that
	\begin{equation}
	\label{condi1}| A(x,\xi)|\le \Lambda |\xi|^{p-1}, \quad | D_\xi A(x,\xi)|\le \Lambda |\xi|^{p-2},
	\end{equation}
	\begin{equation}
	\label{condi2}  \langle D_\xi A(x,\xi)\eta,\eta\rangle\geq \Lambda^{-1}  |\xi|^{p-2} |\eta|^2
	\end{equation}
	for every $x\in \mathbb{R}^n$ and every $(\xi,\eta)\in \mathbb{R}^n\times \mathbb{R}^n\backslash\{(0,0)\}$. More regularity assumptions on function $x\mapsto A(x,\xi)$
	will be needed later.
	\vspace{0.15cm}
	\\
	A typical example   of  \eqref{quasi-measure} is  the $p$-Laplace equation with measure data
	\begin{equation}\label{p-laplace}
	-\Delta_p \, u:= -{\rm div}(|\nabla u|^{p-2} \nabla u)=\mu \quad \text{in~} \Omega.
	\end{equation}
	Since the seminal work of Kilpel\"ainen and Mal\'y  \cite{KM} (see also \cite{TW} for a different approach), the study of pointwise behaviors of solutions to quasilinear equations with measure data \eqref{quasi-measure} has undergone substantial progress.  In particular, the series of works \cite{Duzamin2,55DuzaMing, KM2} (see also \cite{M11}) provide interesting pointwise bounds for gradients of solutions to the seemingly unwieldy  equation \eqref{quasi-measure}, at least for $p>2-\frac{1}{n}$. These pointwise gradient bounds have been extended recently in \cite{HP2, DZ, HP3} for the more singular case $1<p\leq 2- \frac{1}{n}$. 
	\vspace{0.15cm}
	\\
		On the other hand, a more unified approach to pointwise bounds for solutions and their gradients was presented in \cite{KM1}. The results of \cite{KM1} give pointwise bounds
	 not only for the  size but also for the oscillation of solutions and their derivatives expressed in terms of bounds by linear or nonlinear potentials in certain Calder\'on spaces. These cover different kinds of pointwise fractional derivative estimates as well as estimates for  (sharp) fractional maximal functions  of the solutions and their gradients.\vspace{0.15cm}
	\\However, the treatment of \cite{KM1} is still confined to the range $p> 2-\frac{1}{n}$, and the purpose of this note is to extend it to the singular case $1<p\leq 2-\frac{1}{n}$.  Note that, for $1<p\leq 2-\frac{1}{n}$, by looking at the fundamental solution we see that in general distributional solutions of \eqref{p-laplace}  may not even belong to $W^{1,1}_{\rm loc}(\Omega)$. 
\vspace{0.15cm}
\\	
	Thus in this paper we shall restrict ourselves only to the case 
	$$1<p\leq 2-\frac{1}{n},$$
	and note that the main results obtained here also hold in the case $2-\frac{1}{n}<p < 2$ thanks to \cite{KM1}. Moreover, except for the comparison estimates obtained earlier in \cite{HP1, HP3}, the methods used in this paper are very much guided by those of \cite{KM1}. We would also like to point out that there are  analogous results in the case $p\ge 2$ that we  refer to \cite{KM1} for the precise statements.  \\
	In some sense our pointwise regularity for the non-homogeneous equation \eqref{quasi-measure} is obtained from  perturbation/interpolation arguments involving the associated homogeneous equations.  Thus information on the  regularity of associated homogeneous equations will play an important role. In this direction, we first recall a quantitative version 
	of the well-known De Giorgi's result that established  $C^{\alpha_0}$, $\alpha_0\in (0,1)$, regularity for solutions of ${\rm div}\left( {A(x,\nabla w)} \right) = 0$.
	Henceforth, by $Q_r(x_0)$  we mean the open cube $Q_r(x_0):= x_0+ (-r,r)^n$ with center $x_0\in\RR^n$ and side-length  $2r$. In other words,
	$$Q_r(x_0)=\{x\in\RR^n: |x-x_0|_\infty:= \max_{1\leq i\leq n} |x_i-x_{0i}| <r\}.$$
	
	 \begin{lemma}\label{osckkk}
		Under \eqref{condi1}--\eqref{condi2}, let $w\in W^{1,p}(\Om)$, $p>1$, be a solution of the equation ${\rm div}\left( {A(x,\nabla w)} \right) = 0$ in $\Om$. Then there exists 
		$\alpha_0\in (0,1)$, depending only on $n,p$ and $\Lambda$,   such that  for any cubes $Q_\rho(x_0)\subset Q_R(x_0)\subset\Om$, and   $\ep\in (0,1]$, we have 
		\begin{equation}\label{oscp}
		\fint_{Q_\rho(x_0)} |w-(w)_{Q_\rho(x_0)}|^p dx \lesssim \left(\frac{\rho}{R}\right)^{\alpha_0 p} \fint_{Q_R(x_0)} |w-(w)_{Q_R(x_0)}|^p dx,
		\end{equation}  
		and 
		\begin{equation}\label{oscqk}
		\inf_{q\in \RR}\fint_{Q_\rho(x_0)} |w-q|^\ep dx \lesssim \left(\frac{\rho}{R}\right)^{\alpha_0 \ep} \inf_{q\in\RR}\fint_{Q_R(x_0)} |w-q|^\ep dx.
		\end{equation} 
	\end{lemma}
	We point out that 	the proof of \eqref{oscp} follows from \cite[Chapter 7]{Giu}, whereas the proof of \eqref{oscqk} follows from \eqref{oscp} and the reverse H\"older property of $w$.
	\vspace{0.15cm}
	\\
	In the case the nonlinearity $A(x,\xi)$ is independent of $x$, we actually have $C^{1, \beta_0}$, $\beta_0\in (0,1)$, regularity the the homogeneous equation (see, e.g., \cite{DiB, Lin, Man}). For our purpose, we shall use the following quantitative version of this regularity result (see \cite{55DuzaMing, DZ}).
	
	\begin{lemma}\label{osckkkfornab}
		Let $v\in W^{1,p}(\Om)$, $p>1$, be a solution of $\operatorname{div}\left( {A_0(\nabla v)} \right) = 0$ in $\Om$, where $A_0(\xi)$ satisfies \eqref{condi1}--\eqref{condi2}
		and is independent of $x$.  Then there exists 
		$\beta_0\in (0,1)$, depending only on $n,p$ and $\Lambda$, such that  for any cubes $Q_\rho(x_0)\subset Q_R(x_0)\subset\Om$ and  $\ep\in (0,1]$, we have  
		\begin{equation*}
		\fint_{Q_\rho(x_0)} |\nabla w-(\nabla w)_{Q_\rho(x_0)}| dx \lesssim \left(\frac{\rho}{R}\right)^{\beta_0} \fint_{Q_R(x_0)} |\nabla w-(\nabla w)_{Q_R(x_0)}| dx,
		\end{equation*}  
		and
		\begin{equation}\label{oscqk'}
		\inf_{\mathbf{q}\in \RR^n}\fint_{Q_\rho(x_0)} |\nabla v-\mathbf{q}|^\ep dx \lesssim \left(\frac{\rho}{R}\right)^{\beta_0 \ep} \inf_{\mathbf{q}\in\RR^n}\fint_{Q_R(x_0)} |\nabla v-\mathbf{q}|^\ep dx.
		\end{equation}  
	\end{lemma}
	In what follows, we shall use the (maximal) constants $\alpha_0$ in Lemma    \ref{osckkk} and $\beta_0$ in Lemma \ref{osckkkfornab} as certain thresholds in our regularity theory.  
	Also, henceforth, we reserve the letter $\kappa$ for the following constant 
	\begin{equation}\label{kappaconst}
	\kappa:= (p-1)^2/2.
	\end{equation}	
	Our first result provides a De Giorgi's theory for non-homogeneous equations with measure data, which also includes \cite[Theorem 1.4]{HP3}  as an end-point case.
	For the case $p>2-1/n$, see \cite[Theorem 1.1]{KM1}.
	
	\begin{theorem} \label{uptoalpha0}
		Under \eqref{condi1}--\eqref{condi2}, with $1<p \leq 2- \frac{1}{n}$,  let  $\kappa$ be as in \eqref{kappaconst}, and suppose that $u\in C^0(\Omega)\cap W^{1,p}_{\rm loc}(\Om)$ is a solution of \eqref{quasi-measure}. Let $Q_R(x_0)\subset\Omega$ and $\bar{\alpha} \in (0, \alpha_0)$, where $\alpha_0$ is as in Lemma \ref{osckkk}. Then for any $x,y\in Q_{R/8}(x_0)$ 	we have 
		\begin{align}\label{fracalphau}
		|u(x)-u(y)| &\lesssim \left[{\bf W}_{1-\alpha (p-1)/p ,p}^{R}(|\mu|)(x) +{\bf W}_{1-\alpha (p-1)/p ,p}^{R}(|\mu|)(y)\right] |x-y|^\alpha\nonumber\\
		&\qquad  +\left(\fint_{Q_{R}(x_0)} |u|^\kappa dx\right)^{\frac{1}{\kappa}} \left(\frac{|x-y|}{R}\right)^{\alpha}
		\end{align}
		uniformly in $\alpha\in [0,\bar{\alpha}]$. Here the implicit constant depends only on $n,p,\Lambda$, and $\bar{\alpha}$.
	\end{theorem}
In \eqref{fracalphau}, the function ${\bf W}_{1-\alpha (p-1)/p ,p}^{R}(|\mu|)(\cdot)$ is a truncated Wolff's potential of $|\mu|$.  In general, given a nonnegative measure $\nu$ and 
$\rho>0$, the Wolff's potential ${\bf W}^{\rho}_{\alpha, s} \nu$, $\alpha>0, s>1$, is defined by 
$${\bf W}^{\rho}_{\alpha,s} (\nu)(x):=\int_{0}^{\rho} \left[\frac{\nu(Q_t(x))}{t^{n-\alpha s}}\right]^{\frac{1}{s-1}} \frac{dt}{t}, \qquad x\in\RR^n.$$ 
Note that 	${\bf W}^{\rho}_{\alpha, 2} (\nu)= {\bf I}^{\rho}_{2\alpha} (\nu)$, where ${\bf I}^{\rho}_{\gamma} (\nu)$, $\gamma >0$, is a truncated Riesz's potential defined by
$${\bf I}^{\rho}_{\gamma}( \nu)(x):=\int_{0}^{\rho} \frac{\nu(Q_t(x))}{t^{n-\gamma}} \frac{dt}{t}, \qquad x\in\RR^n.$$ 
\vspace{0.15cm}
\\
We remark  that, except for 	\eqref{condi1}--\eqref{condi2}, no further regularity assumption is needed in Theorem \ref{uptoalpha0}. However, this will force the constant $\bar{\alpha}$ to be small in general. 
	\vspace{0.15cm}
	\\
	On the other hand, it is possible to allow $\bar{\alpha}$ to be arbitrarily close to $1$ as long as we further impose a `small BMO' condition on the map
	$x\mapsto A(x,\xi)$. This condition entails the smallness of the limit $\limsup_{\rho\rightarrow 0} \om(\rho)$, where 
	\begin{equation}\label{omdef}
	\om(\rho):=\sup_{y\in\RR^n}\left[\fint_{Q_{r}(y)}\Upsilon(A,
	Q_{r}(y))(x)^{2} dx\right]^{\frac{1}{2}}, \qquad \rho >0,
	\end{equation}
	and for each cube $Q_r(y)$ we set
	\[
	\Upsilon(A,Q_r(y))(x) := \sup_{\xi \in \mathbb{R}^{n}\setminus \{0\}} \frac{|A({x, \xi}) - \overline{A}_{Q_r(y)}({\xi})|}{|\xi|^{p-1}},
	\]
	with $\overline{A}_{Q_r(y)}(\xi) = \fint_{Q_r(y)}A(x, \xi)dx$.
	The precise statement is as follows.
	\begin{theorem}  \label{upto1}
		Under \eqref{condi1}--\eqref{condi2}, with $1<p \leq 2- \frac{1}{n}$,  let  $\kappa$ be as in \eqref{kappaconst}, and suppose that  $u\in C^0(\Om)$ is a solution to \eqref{quasi-measure}. Let $Q_{R}(x_0)\subset\Om$. Then for any positive $\bar{\alpha}<1$
		there exists a small $\delta=\delta(n,p,\Lambda, \bar{\alpha})>0$ such that if 
		\begin{equation}\label{smallBMO}
	\limsup_{\rho\rightarrow 0} \om(\rho)\leq\delta,
		\end{equation}
		then 	 for any $x,y\in Q_{R/8}(x_0)\subset\Om$, we have 	 
		\begin{align}\label{fracalphau-new}
		|u(x)-u(y)| &\lesssim \left[{\bf W}_{1-\alpha (p-1)/p ,p}^{R}(|\mu|)(x) +{\bf W}_{1-\alpha (p-1)/p ,p}^{R}(|\mu|)(y)\right] |x-y|^\alpha\nonumber\\
		&\qquad  +\left(\fint_{Q_{R}(x_0)} |u|^\kappa dx\right)^{\frac{1}{\kappa}} \left(\frac{|x-y|}{R}\right)^{\alpha}
		\end{align}
		uniformly in $\alpha\in [0,\bar{\alpha}]$. Here the implicit constant depends on $n,p,\Lambda, \bar{\alpha},\om(\cdot)$, and ${\rm diam}(\Om)$.
	\end{theorem}  
	Under a certain Dini-VMO condition, we could also allow $\bar{\alpha}=1$ in the above theorem. However, in this case the Wolff's potential is replaced with a Riesz's potential 
	raised to the power of $\frac{1}{p-1}$.
	
	\begin{theorem}\label{0toi1}  Under \eqref{condi1}--\eqref{condi2}, with $1<p \leq 2- \frac{1}{n}$,  let  $\kappa$ be as in \eqref{kappaconst}, and suppose that  $u\in C^1(\Om)$ is a solution to \eqref{quasi-measure}. Let $Q_{R}(x_0)\subset\Om$. If for some $\sigma_1\in (0,1)$ such that $\om(\cdot)^{\sigma_1}$ is Dini-VMO, i.e., 
		\begin{equation}\label{Dini-VMO}
		\int_{0}^{1} \om(\rho)^{\sigma_1} \frac{d\rho}{\rho}<+\infty,
		\end{equation}
		then  for any $x,y\in Q_{R/8}(x_0)\subset\Om$, we have 	 
		\begin{align*}
		&|u(x)-u(y)| \\
		&\quad \lesssim \left[\left({\bf I}_{p-\alpha (p-1)}^{R}(|\mu|)(x)\right)^{\frac{1}{p-1}} + \left({\bf I}_{p-\alpha (p-1)}^{R}(|\mu|)(y)\right)^{\frac{1}{p-1}}\right] |x-y|^\alpha\nonumber\\
		&\quad  +\left(\fint_{Q_{R}(x_0)} |u|^\kappa dx\right)^{\frac{1}{\kappa}} \left(\frac{|x-y|}{R}\right)^{\alpha}
		\end{align*}
		uniformly in $\alpha \in [0, 1]$. Here the implicit constant depends on $n,p,\Lambda, \bar{\alpha}, \sigma_1, \om(\cdot)$, and ${\rm diam}(\Om)$.
	\end{theorem}  
	We remark that, when $\alpha=1$, Theorem \ref{0toi1} recovers the pointwise gradient estimates of \cite{DZ} and \cite{HP3}  that were obtained under a slightly different Dini condition.
	\vspace{0.15cm}
	\\
	Finally, under a stronger Dini-H\"older condition we can also bound solution gradients in  appropriate Calder\'on spaces.
	
	\begin{theorem}\label{gradholdTHM} Under \eqref{condi1}--\eqref{condi2}, with $1<p \leq 2- \frac{1}{n}$,  let  $\kappa$ be as in \eqref{kappaconst}, and suppose that $u\in C^1(\Om)$ is a solution to \eqref{quasi-measure}. Let $Q_{R}(x_0)\subset\Om$. If for some $\sigma_1\in (0,1)$ such that
	$\om(\cdot)^{\sigma_1}$ is 	Dini-H\"older of order $\bar{\alpha}$, i.e.,
		\begin{equation}\label{dhcond}
		  \int_{0}^1\frac{\om(\rho)^{\sigma_1}}{\rho^{\bar{\alpha}}} \frac{d\rho}{\rho} <+\infty
		\end{equation}
		for some $\bar{\alpha}\in [0,\beta_0)$, 
		then for any $x,y\in Q_{R/4}(x_0)\subset\Om$, we have 	 
		\begin{align*}
		|\nabla u(x)- \nabla u(y)|  &\lesssim \left[\left({\bf I}_{1-\alpha}^{R}(|\mu|)(x)\right)^{\frac{1}{p-1}} + \left({\bf I}_{1-\alpha}^{R}(|\mu|)(y)\right)^{\frac{1}{p-1}}\right] |x-y|^\alpha\nonumber\\
		&\qquad \quad  +\left(\fint_{Q_{R}(x_0)} |\nabla u|^\kappa dx\right)^{\frac{1}{\kappa}} \left(\frac{|x-y|}{R}\right)^{\alpha}
		\end{align*}
		uniformly in $\alpha \in [0, \bar{\alpha}]$. Here $\beta_0$ is as  in Lemma \ref{osckkkfornab}, and  the implicit constant depends on $n,p,\Lambda, \bar{\alpha}, \sigma_1, \om(\cdot)$, and ${\rm diam}(\Om)$. 
	\end{theorem}

	\section{Comparison and Poincar\'e type inequalities}
	
	The study of  regularity problems for equation \eqref{quasi-measure} is based on the following comparison
	estimate that connects the solution 
	of measure datum problem to a solution of a homogeneous problem. 
	\vspace{0.15cm}\\
	To describe it,  we let  $u\in W_{\rm loc}^{1,p}(\Omega)$ be a solution of \eqref{quasi-measure}.
	Then for  a cube $Q_{2R}=Q_{2R}(x_0) \Subset\Omega$, we   consider the  unique solution $w\in W_0^{1,p}(Q_{2R}(x_0))+u$ to the local interior problem

	\begin{equation}\label{eq1}
	\left\{ \begin{array}{rcl}
	- \operatorname{div}\left( {A(x,\nabla w)} \right) &=& 0 \quad \text{in} \quad Q_{2R}(x_0), \\ 
	w &=& u\quad \text{on} ~~ \partial Q_{2R}(x_0).  
	\end{array} \right.
	\end{equation}

	\begin{lemma}\label{maininterior} Suppose that  $Q_{3 R}(x_0)\subset\Om$ for some $R>0$. Let $u$ and $w$ be as in \eqref{eq1} and let  $\kappa$ be as in \eqref{kappaconst}, where $1<p \leq 2-\frac{1}{n}$. Then it holds that 
		\begin{align}\label{compgradu}
		\left(	\fint_{Q_{2R}(x_0)}|\nabla (u-w)|^{\ka} dx\right)^{\frac{1}{\kappa}} & \lesssim \left(\frac{|\mu|(Q_{3 R}(x_0))}{R^{n-1}} \right)^{\frac{1}{p-1}} \nonumber\\
		&\qquad  + \frac{|\mu|(Q_{3 R}(x_0))}{R^{n-1}}	
		\left(	\fint_{Q_{3 R}(x_0)}|\nabla u|^{\ka} dx \right)^{\frac{2-p}{\kappa}}.
		\end{align} 
	\end{lemma}
\begin{proof} For $1<p\leq \frac{3n-2}{2n-1}$, inequality \eqref{compgradu} was obtained in \cite[Theorem 1.2]{HP3}. For $\frac{3n-2}{2n-1}<p\leq 2-\frac{1}{n}$, 
	by \cite[Lemma 2.2]{HP1}, we have 
	\begin{align*}
	\left(	\fint_{Q_{2R}(x_0)}|\nabla (u-w)|^{\gamma_0} dx \right)^{\frac{1}{\gamma_0}} & \lesssim \left(\frac{|\mu|(Q_{2 R}(x_0))}{R^{n-1}}\right)^{\frac{1}{p-1}} \nonumber\\
	&\qquad  + \frac{|\mu|(Q_{2 R}(x_0))}{R^{n-1}}	
	\left(	\fint_{Q_{2 R}(x_0)}|\nabla u|^{\gamma_0} dx \right)^{\frac{2-p}{\gamma_0}}
	\end{align*}
for some $\gamma_0 \in [\frac{2-p}{2}, \frac{n(p-1)}{n-1})$.	In fact, an inspection of the proof of \cite[Lemma 2.2]{HP1} reveals that we can take {\it any}
 $\gamma_0 \in (\frac{n}{2n-1}, \frac{n(p-1)}{n-1})$. Thus we may assume that 
$\kappa=(p-1)^2/2 < \gamma_0$. To conclude the proof,  it is therefore enough to show that 
\begin{equation}\label{enough}
\left(	\fint_{Q_{2 R}(x_0)}|\nabla u|^{\gamma_0} dx \right)^{\frac{1}{\gamma_0}} \lesssim \left(\frac{|\mu|(Q_{3 R}(x_0))}{R^{n-1}}\right)^{\frac{1}{p-1}}+ \left(	\fint_{Q_{3 R}(x_0)}|\nabla u|^{\ka} dx \right)^{\frac{1}{\ka}}.
\end{equation}\\
To this end, let $\gamma_1\in (\gamma_0, \frac{n(p-1)}{n-1})$. By \cite[Corollay 2.4]{HP3}, we have 
\begin{align}\label{nabuga}
&\left(	 \fint_{Q_{ \rho}(x)}|\nabla u|^{\gamma_1} dy \right)^{\frac{1}{\gamma_1}} \nonumber\\
&\quad \lesssim \left(\frac{|\mu|(Q_{9\rho/8}(x))}{\rho^{n-1}}\right)^{\frac{1}{p-1}} + \frac{1}{\rho} \left(\fint_{Q_{9\rho/8}(x)}|u-\lambda|^{\gamma_0} dy\right)^{\frac{1}{\gamma_0}}
\end{align}
for any $\lambda\in\RR$ and any cube $Q_\rho(x)$ such that $Q_{9\rho/8}(x)\subset\Om$. \\
Now, with $Q_{8\rho/7}(x)\subset\Om$, let $w_1$ be the  unique solution $w_1\in W_0^{1,p}(Q_{8\rho/7}(x))+u$ to the  problem
\begin{equation*}
\left\{ \begin{array}{rcl}
- \operatorname{div}\left( {A(x,\nabla w_1)} \right) &=& 0 \quad \text{in} \quad Q_{8\rho/7}(x), \\ 
w_1 &=& u\quad \text{on} \quad \partial Q_{8\rho/7}(x).  
\end{array} \right.
\end{equation*}
Then from the proof of \cite[Lemma 2.2]{HP1} (using (2.8) and (2.18) in \cite{HP1}), we can deduce that 

\begin{align}\label{uw1nab}
 \frac{1}{\rho} \left(\fint_{Q_{8\rho/7}(x)}|u-w_1|^{\gamma_0} dy\right)^{\frac{1}{\gamma_0}} & \lesssim \left(\frac{|\mu|(Q_{8\rho/7}(x))}{\rho^{n-1}}\right)^{\frac{1}{p-1}} \nonumber\\
 & \quad +
\frac{|\mu|(Q_{8\rho/7}(x))}{\rho^{n-1}} \left(	\fint_{Q_{8\rho/7}(x)}|\nabla u|^{\gamma_0} dy \right)^{\frac{2-p}{\gamma_0}}.
\end{align}
By Young's inequality, this yields 
\begin{align} \label{uw1}
&\frac{1}{\rho} \left(\fint_{Q_{8\rho/7}(x)}|u-w_1|^{\gamma_0} dy\right)^{\frac{1}{\gamma_0}} \nonumber\\
 &\qquad  \lesssim \left(\frac{|\mu|(Q_{8\rho/7}(x))}{\rho^{n-1}}\right)^{\frac{1}{p-1}} +
 \left(	\fint_{Q_{8\rho/7}(x)}|\nabla u|^{\gamma_0} dy \right)^{\frac{1}{\gamma_0}}.
\end{align}
Thus by quasi-triangle and H\"older's inequalities  we get 
\begin{align}\label{ugamma}
&\frac{1}{\rho} \left(\fint_{Q_{9\rho/8}(x)}|u-\lambda|^{\gamma_0} dy\right)^{\frac{1}{\gamma_0}} \nonumber\\
&\qquad \lesssim 
\left(\frac{|\mu|(Q_{8\rho/7}(x))}{\rho^{n-1}}\right)^{\frac{1}{p-1}} +
\left(	\fint_{Q_{8\rho/7}(x)}|\nabla u|^{\gamma_0} dy \right)^{\frac{1}{\gamma_0}}\nonumber\\
& \qquad + \frac{1}{\rho} \left(\fint_{Q_{9\rho/8}(x)}|w_1-\lambda|^\frac{n}{n-1} dy\right)^{\frac{n-1}{n}},
\end{align}
where we choose $\lambda = \fint_{Q_{9\rho/8}(x)} w_1 dz$.\\
We now use Poincar\'e and the reverse H\"older's inequalities for $\nabla w_1$ to obtain that 
\begin{align*}
 &\frac{1}{\rho} \left(\fint_{Q_{9\rho/8}(x)}|w_1-\lambda|^\frac{n}{n-1} dy\right)^{\frac{n-1}{n}}\\
 &\qquad  \lesssim  \fint_{Q_{9\rho/8}(x)}|\nabla w_1| dy
 \lesssim  \left(\fint_{Q_{8\rho/7}(x)}|\nabla w_1|^{\gamma_0} dy\right)^{\frac{1}{\gamma_0}}\\
 &\qquad \lesssim \left(\fint_{Q_{8\rho/7}(x)}|\nabla u-\nabla w_1|^{\gamma_0} dy\right)^{\frac{1}{\gamma_0}} + \left(\fint_{Q_{8\rho/7}(x)}|\nabla u|^{\gamma_0} dy\right)^{\frac{1}{\gamma_0}}\\
 &\qquad \lesssim \left(\frac{|\mu|(Q_{8\rho/7}(x))}{\rho^{n-1}}\right)^{\frac{1}{p-1}} + \left(\fint_{Q_{8\rho/7}(x)}|\nabla u|^{\gamma_0} dy\right)^{\frac{1}{\gamma_0}},
\end{align*}
where we used \cite[Lemma 2.2]{HP1} and Young's inequality in the last bound.\vspace{0.15cm}
\\	
Thus combining this result with \eqref{ugamma} we find 	
\begin{align*}
\frac{1}{\rho} \left(\fint_{Q_{9\rho/8}(x)}|u-\lambda|^{\gamma_0} dy\right)^{\frac{1}{\gamma_0}} & \lesssim 
\left(\frac{|\mu|(Q_{8\rho/7}(x))}{\rho^{n-1}}\right)^{\frac{1}{p-1}} +
\left(	\fint_{Q_{8\rho/7}(x)}|\nabla u|^{\gamma_0} dy \right)^{\frac{1}{\gamma_0}}.
\end{align*}	
At this point, plugging this into \eqref{nabuga} we arrive at 
\begin{align*}
\left(	 \fint_{Q_{ \rho}(x)}|\nabla u|^{\gamma_1} dy \right)^{\frac{1}{\gamma_1}} \quad \lesssim \left(\frac{|\mu|(Q_{8\rho/7}(x))}{\rho^{n-1}}\right)^{\frac{1}{p-1}} + \left(	\fint_{Q_{8\rho/7}(x)}|\nabla u|^{\gamma_0} dy \right)^{\frac{1}{\gamma_0}}, 
\end{align*}
which holds for any cube $Q_{ \rho}(x)$ such that $Q_{ 8\rho/7}(x)\subset\Om$. Recall that $\gamma_1>\gamma_0$, and 
thus by  a covering/iteration argument as in \cite[Remark 6.12]{Giu}, we have 
\begin{align}\label{reverseHtype}
\left(	 \fint_{Q_{ \rho}(x)}|\nabla u|^{\gamma_1} dy \right)^{\frac{1}{\gamma_1}} \quad \lesssim \left(\frac{|\mu|(Q_{8\rho/7}(x))}{\rho^{n-1}}\right)^{\frac{1}{p-1}} + \left(	\fint_{Q_{8\rho/7}(x)}|\nabla u|^{\ep} dy \right)^{\frac{1}{\ep}}
\end{align}
for any $\epsilon>0$. This obviously yields \eqref{enough} as desired and the proof is complete.
\end{proof}
	
\begin{remark}\label{RemarkCOM} Using the above argument, in particular \eqref{uw1nab}, we can also show the following comparison estimate for the functions $u$ and $w$: for any $\frac{3n-2}{2n-1}<p\leq 2-\frac{1}{n}$,
\begin{align*}
\left(	\fint_{Q_{2R}(x_0)}|u-w|^{\ka} dx\right)^{\frac{1}{\kappa}} & \lesssim \left(\frac{|\mu|(Q_{3 R}(x_0))}{R^{n-p}} \right)^{\frac{1}{p-1}} \nonumber\\
&\qquad  + \frac{|\mu|(Q_{3 R}(x_0))}{R^{n-2}}	
\left(	\fint_{Q_{3 R}(x_0)}|\nabla u|^{\ka} dx \right)^{\frac{2-p}{\kappa}},
\end{align*}	
and 
\begin{align*}
\left(	\fint_{Q_{2R}(x_0)}|u-w|^{\ka} dx\right)^{\frac{1}{\kappa}} & \lesssim \left(\frac{|\mu|(Q_{3 R}(x_0))}{R^{n-p}} \right)^{\frac{1}{p-1}} \nonumber\\
&\qquad  + \frac{|\mu|(Q_{3 R}(x_0))}{R^{n-p}}	
\left(	\fint_{Q_{3 R}(x_0)}|u-\lambda|^{\ka} dx \right)^{\frac{2-p}{\kappa}}
\end{align*}	
for any $\lambda\in\RR$. For $1<p\leq \frac{3n-2}{2n-1}$, these inequalities have  been obtained in \cite[Theorem 1.2]{HP3}.
\end{remark}	
The following Poincar\'e type inequality  was obtained in the case $1<p\leq \frac{3n-2}{2n-1}$ in \cite[Corollary 1.3]{HP3}. A similar proof using  Lemma \ref{maininterior}   and inequalities of the form \eqref{uw1} and \eqref{reverseHtype} also yields the result in the case $\frac{3n-2}{2n-1}< p\leq 2-\frac{1}{n}.$
\begin{corollary}\label{poicaresub} Suppose that  $Q_{3r/2}(x_0)\subset\Om$ for some $r>0$. 
Let  $u\in W^{1,p}_{\rm loc}(\Om)$, $1<p\leq 2- \frac{1}{n}$, be a solution of \eqref{quasi-measure}. Then for any $\epsilon>0$ we have   
	\begin{align*}
		\inf_{q\in \RR} \left(\fint_{Q_r(x_0)}|u-q|^{\ep}\right)^{\frac{1}{\ep}} \lesssim\left(\frac{|\mu|(Q_{3r/2}(x_0))}{r^{n-p}}\right)^{\frac{1}{p-1}}+ 	
	r \left(	\fint_{Q_{ 3r/2}(x_0)}|\nabla u|^{\ep}\right)^{\frac{1}{\ep}}.
	\end{align*}
\end{corollary}		
With $u$ and $w$ as in \eqref{eq1}, we now consider another auxiliary function $v$	such that $v\in W^{1,p}_{0}(Q_R(x_0))+w$ is the unique solution to the equation
\begin{equation}\label{Aux2}
\left\{ \begin{array}{rcl}
- \operatorname{div}\left( \overline{A}_{Q_R(x_0)}(\nabla v) \right) &=& 0 \quad \text{in} \quad Q_{R}(x_0), \\ 
v &=& w\quad \text{on} \quad \partial Q_{R}(x_0),  
\end{array} \right.
\end{equation}
where  $\overline{A}_{Q_R(x_0)}(\xi) = \fint_{Q_R(x_0)}A(x, \xi)dx$.
\\The following result can be deduced from \cite[Lemma 2.3]{KM1} and an appropriate reverse H\"older's inequality. 	
	
\begin{lemma}\label{vwsigma0} Let $p>1$, $0<\ep\leq p$, and $u, w$, and $v$ be as in \eqref{eq1} and \eqref{Aux2}, where $Q_{2R}(x_0)\Subset\Om$. Then there exists a small positive constant $\sigma_0>0$ such that 
$$	\left(\fint_{Q_R(x_0)} |\nabla v - \nabla w|^\ep dx \right)^{\frac{1}{\ep}} \lesssim \om(R)^{\sigma_0}  \left(\fint_{Q_{2R}(x_0)} |\nabla w|^\ep dx \right)^{\frac{1}{\ep}},
$$	
where $\om(\cdot)$ is as defined in \eqref{omdef}.
\end{lemma}	
	Likewise, following lemma follows from \cite[Lemma 2.5]{KM1}.
\begin{lemma}\label{Dini-VMOlem} Let $1<p<2$, $0<\ep\leq p$, and $u, w$, and $v$ be as in \eqref{eq1} and \eqref{Aux2}, where $Q_{2R}(x_0)\Subset\Om$. Then
for any $\sigma_1\in (0,1)$ such that $\om(\cdot)^{\sigma_1}$ is Dini-VMO, i.e., \eqref{Dini-VMO} holds,
it follows that 
	$$	\left(\fint_{Q_R(x_0)} |\nabla v - \nabla w|^\ep dx \right)^{\frac{1}{\ep}} \lesssim \om(R)^{\sigma_1}  \left(\fint_{Q_{2R}(x_0)} |\nabla w|^\ep dx \right)^{\frac{1}{\ep}}.
	$$	
\end{lemma}

\section{Pointwise fractional maximal function bounds}

As in \cite{KM1}, our proofs of Theorems \ref{uptoalpha0}--\ref{gradholdTHM} are based on the corresponding pointwise estimates for the associate  fractional and sharp fractional maximal functions, which are interesting in their own right. This section is devoted to such pointwise fractional maximal function bounds.\\
Given $R>0$ and $q>0$, following \cite{DS}, we define the following truncated sharp fractional maximal function of a function $f\in L^q_{\rm loc}(\RR^n)$:

$${\bf M}^{\#, R}_{\alpha, q}(f)(x):= \sup_{0<\rho\leq R} \inf_{m\in \RR} \rho^{-\alpha}\left(\fint_{Q_\rho(x)} |f-m|^q dx\right)^{\frac{1}{q}}, \qquad \alpha\geq 0.$$
Also,  we define a truncated fractional maximal function by
$${\bf M}^{R}_{\beta, q}(f)(x):= \sup_{0<\rho\leq R}  \rho^{\beta}\left(\fint_{Q_\rho(x)} |f|^q dx\right)^{\frac{1}{q}}, \qquad \beta\in [0,n/q].$$
In the case $q=1$, we usually drop the index $q$ in the above notation, i.e., we set ${\bf M}^{\#, R}_{\alpha, 1}(f)={\bf M}^{\#, R}_{\alpha}(f)$
and  ${\bf M}^{R}_{\beta, 1}(f)={\bf M}^{R}_{\beta}(f)$. Moreover, the definition of ${\bf M}^{R}_{\beta}(f)$ can also be naturally extended to the case where  $f=\mu$ is a locally finite signed measure in $\RR^n$:
 $${\bf M}^{R}_{\beta}(\mu)(x):= \sup_{0<\rho\leq R}  \rho^{\beta}\frac{|\mu|(Q_\rho(x))}{|Q_\rho(x)|}, \qquad \beta\in [0,n/q].$$
Note that  by  Poincar\'e  inequality we have 
$${\bf M}^{\#, R}_{\beta}(f)(x)\lesssim {\bf M}^{R}_{1-\beta}(\nabla f)(x), \qquad \beta \in[0,1],$$
for any $f\in W^{1,1}_{\rm loc}(\RR^n)$.\\
On the other hand, if $u\in W^{1,p}_{\rm loc}(\Om)$, $1<p\leq 2-\frac{1}{n}$, then it follows from Corollary \ref{poicaresub} that   
\begin{equation}\label{sharppoinc}
{\bf M}^{\#,R}_{\beta, \ep}(u)(x)\lesssim  \left[{\bf M}^{3R/2}_{p-\beta(p-1)}(\mu)(x)\right]^{\frac{1}{p-1}}+ {\bf M}^{3R/2}_{1-\beta, \ep}(\nabla u)(x),  \beta \in[0,1],
\end{equation}
for any $\ep \in (0,1)$ and any cube $Q_{3R/2}(x)\subset\Om$.\vspace{0.1cm}\\
The following fractional maximal function bound will be  needed in the proof of Theorem \ref{uptoalpha0}. 
\begin{theorem} \label{sharpmax}
 	Under \eqref{condi1}--\eqref{condi2}, let $1<p \leq 2-\frac{1}{n}$,  and suppose that $u\in W^{1,p}_{\rm loc}(\Om)$ is a solution of \eqref{quasi-measure}. Let $Q_{3R}(x)\subset\Omega$ and $\bar{\alpha} \in (0, \alpha_0)$, where $\alpha_0\in (0,1)$ is as in Lemma \ref{osckkk}. Then 	we have 
	\begin{align}\label{sharpmaxu}
	& {\bf M}^{\#,2R}_{\alpha, \kappa}(u)(x)+ {\bf M}^{3R}_{1-\alpha, \kappa}(\nabla u)(x) \nonumber\\
	&\qquad \lesssim  \left [ {\bf M}^{3R}_{p-\alpha(p-1)}(\mu)(x) \right]^{\frac{1}{p-1}} + R^{1-\alpha} \left(\fint_{Q_{3R}(x)} |\nabla u|^\kappa dy\right)^\frac{1}{\kappa} 
	\end{align}
	uniformly in $\alpha\in [0,\bar{\alpha}]$. Here the implicit constant depends on $n,p,\Lambda$, and $\bar{\alpha}$.
\end{theorem}
\begin{proof}
	The main idea of the proof of \eqref{sharpmaxu} lies the proof of  \cite[Proposition 3.1]{KM1} that treated the case $p>2-\frac{1}{n}$.  Note that  by \eqref{sharppoinc} 
	it is enough to show 
	\begin{align}\label{sharpmaxugrad}
	{\bf M}^{\ep R}_{1-\alpha, \kappa}(\nabla u)(x) \lesssim  \left [ {\bf M}^{3R}_{p-\alpha(p-1)}(\mu)(x) \right]^{\frac{1}{p-1}} + R^{1-\alpha} \left(\fint_{Q_{3R}(x)} |\nabla u|^\kappa dy\right)^\frac{1}{\kappa},
	\end{align}
	for some $\ep =\ep_1(n,p,\Lambda, \bar{\alpha})\in (0,1)$.\\
	Let $0<\rho\leq r\leq R$, and choose $w$ as in \eqref{eq1} with $Q_{2r}(x)$ in place of $Q_{2R}(x_0)$. 
	We have
	\begin{align*}
	&\fint_{Q_\rho(x)}|\nabla u|^\kappa dy \lesssim \fint_{Q_\rho(x)}|\nabla w|^\kappa dy + \left(\frac{r}{\rho}\right)^n \fint_{Q_{2r}(x)}|\nabla u-\nabla w|^\kappa dy\\
	& \lesssim \left(\frac{\rho}{r}\right)^{(\alpha_0-1)\ka} \fint_{Q_{2r}(x)}|\nabla w|^\kappa dy + \left(\frac{r}{\rho}\right)^n \fint_{Q_{2r}(x)}|\nabla u-\nabla w|^\kappa dy\\
	&  \lesssim \left(\frac{\rho}{r}\right)^{(\alpha_0-1)\ka} \fint_{Q_{2r}(x)}|\nabla u|^\kappa dy\\
	&\qquad \qquad  + \left\{\left(\frac{\rho}{r}\right)^{(\alpha_0-1)\ka} +\left(\frac{r}{\rho}\right)^n \right\} \fint_{Q_{2r}(x)}|\nabla u-\nabla w|^\kappa dy\\
&  \lesssim \left(\frac{\rho}{r}\right)^{(\alpha_0-1)\ka} \fint_{Q_{2r}(x)}|\nabla u|^\kappa dy + \left(\frac{r}{\rho}\right)^n  \fint_{Q_{2r}(x)}|\nabla u-\nabla w|^\kappa dy,
\end{align*}
	where we used the inequality 
	\begin{equation*}
	\fint_{Q_\rho(x)}|\nabla w|^\kappa dy\lesssim \left(\frac{\rho}{r}\right)^{(\alpha_0-1)\ka} \fint_{Q_{2r}(x)}|\nabla w|^\kappa dy,
	\end{equation*}
which is a modified version of \cite[Theorem  2.2]{KM1}, in the second inequality. \\
Thus by Lemma \ref{maininterior} we get 
	\begin{align*}
	&\left(\fint_{Q_\rho(x)}|\nabla u|^\kappa dy\right)^{1/\kappa}  \lesssim \left(\frac{\rho}{r}\right)^{\alpha_0-1} \left(\fint_{Q_{2r}(x)}|\nabla u|^\kappa dy\right)^{1/\ka} \\
	&\qquad + \left(\frac{r}{\rho}\right)^{n/\ka} \left[\frac{|\mu|(Q_{3r}(x))}{r^{n-1}}\right]^{\frac{1}{p-1}}\\
	&\qquad + \left(\frac{r}{\rho}\right)^{n/\ka} \left(\frac{|\mu|(Q_{3r}(x))}{r^{n-1}}\right) \left( \fint_{Q_{3r}(x)}|\nabla u|^\kappa dy\right)^{(2-p)/\ka}.
	\end{align*}
	Let $\epsilon \in (0,1)$, and choose $\rho=\epsilon r$. Then by Young's inequality we have
\begin{align*}
&\left(\fint_{Q_{\ep r}(x)}|\nabla u|^\kappa dy\right)^{1/\kappa}  \leq  C(\ep) \left[\frac{|\mu|(Q_{3r}(x))}{r^{n-1}}\right]^{\frac{1}{p-1}} \\
& \qquad \qquad + [C\ep^{\alpha_0-1} +1] \left(\fint_{Q_{3r}(x)}|\nabla u|^\kappa dy\right)^{1/\ka}.
\end{align*}
Multiplying both sides by $(\epsilon r)^{1-\alpha}$, $0<\alpha\leq \bar{\alpha} <\alpha_0$, and taking the supremum with respect to $r\in(0,R]$, we find 
\begin{align*}
& \sup_{0<r\leq \ep R}r^{1-\alpha}\left(\fint_{Q_{r}(x)}|\nabla u|^\kappa dy\right)^{1/\kappa}  \leq  C(\ep) \sup_{0<r\leq  R} \left[\frac{|\mu|(Q_{3r}(x))}{r^{n-p+\alpha(p-1)}}\right]^{\frac{1}{p-1}} \\
& \qquad \qquad + [C\ep^{\alpha_0-1} +1] (\ep/3)^{1-\alpha} \sup_{0<r\leq  R} (3r)^{1-\alpha} \left(\fint_{Q_{3r}(x)}|\nabla u|^\kappa dy\right)^{1/\ka}.
\end{align*}
	We now choose $\epsilon \in (0,1)$ such that
	$$ [C   \epsilon^{\alpha_0-1 } +1] (\epsilon/3)^{1-\bar{\alpha}}\leq 1/2,$$
	to deduce that 
\begin{align*}
& \sup_{0<r\leq \ep R}r^{1-\alpha}\left(\fint_{Q_{r}(x)}|\nabla u|^\kappa dy\right)^{1/\kappa}  \leq  C(\ep) \sup_{0<r\leq  R} \left[\frac{|\mu|(Q_{3r}(x))}{r^{n-p+\alpha(p-1)}}\right]^{\frac{1}{p-1}} \\
& \qquad \qquad +  \sup_{\ep R < r\leq  3R} r^{1-\alpha} \left(\fint_{Q_{3r}(x)}|\nabla u|^\kappa dy\right)^{1/\ka}\\
& \qquad  \lesssim \left [ {\bf M}^{3R}_{p-\alpha(p-1)}(\mu)(x) \right]^{\frac{1}{p-1}} + R^{1-\alpha} \left(\fint_{Q_{3R}(x)} |\nabla u|^\kappa dy\right)^\frac{1}{\kappa}.
\end{align*}
		This is \eqref{sharpmaxugrad} and the proof is complete.
\end{proof}\vspace{0.2cm}\\
The following result will be needed for the proof of Theorem \ref{upto1}.
\begin{theorem}\label{smallBMOTHM} Let $1<p\leq 2-\frac{1}{n}$ and  $u\in C^0(\Om)$ be a solution to \eqref{quasi-measure}. Suppose that $Q_{3R}(x)\subset\Om$. Then for any positive $\bar{\alpha}<1$
there exists a small $\delta=\delta(n,p,\Lambda, \bar{\alpha})>0$ such that if \eqref{smallBMO} holds,
 then 	the estimate 
	\begin{align*}
	& {\bf M}^{\#,2R}_{\alpha, \kappa}(u)(x)+ {\bf M}^{3R}_{1-\alpha, \kappa}(\nabla u)(x)\\
	&\qquad \lesssim  \left [ {\bf M}^{3R}_{p-\alpha(p-1)}(\mu)(x) \right]^{\frac{1}{p-1}} + R^{1-\alpha} \left(\fint_{Q_{3R}(x)} |\nabla u|^\kappa dy\right)^\frac{1}{\kappa}
	\end{align*}
	holds uniformly in $\alpha\in[0,\bar{\alpha}]$. Here the implicit constant depends on $n,p,\Lambda, \bar{\alpha},\om(\cdot)$, and ${\rm diam}(\Om)$.
\end{theorem}
\begin{proof}
The proof is similar to that of Theorem \ref{sharpmax}, but this time we need to use  Lemma \ref{vwsigma0}.
As above,  by \eqref{sharppoinc} 
it is enough to show \eqref{sharpmaxugrad} for some $\ep =\ep_1(n,p,\Lambda, \bar{\alpha})\in (0,1)$.
Let $0<\rho\leq r\leq R$, and choose $w$ as in \eqref{eq1} with $Q_{2r}(x)$ in place of $Q_{2R}(x_0)$. Then choose $v$ as in \eqref{Aux2}  with $Q_{r}(x)$ in place of $Q_{R}(x_0)$.
This time we have
\begin{align*}
&\fint_{Q_\rho(x)}|\nabla u|^\kappa dy \\&\lesssim \fint_{Q_\rho(x)}|\nabla v|^\kappa dy + \left(\frac{r}{\rho}\right)^n\fint_{Q_{r}(x)}|\nabla v-\nabla w|^\kappa dy  + \left(\frac{r}{\rho}\right)^n \fint_{Q_{2r}(x)}|\nabla u-\nabla w|^\kappa dy\\
& \lesssim  \fint_{Q_{r}(x)}|\nabla v|^\kappa dy +  \left(\frac{r}{\rho}\right)^n\fint_{Q_{r}(x)}|\nabla v-\nabla w|^\kappa dy + \left(\frac{r}{\rho}\right)^n \fint_{Q_{2r}(x)}|\nabla u-\nabla w|^\kappa dy\\
&  \lesssim  \fint_{Q_{r}(x)}|\nabla u|^\kappa dy + \left\{1+ \left(\frac{r}{\rho}\right)^n \right\}\left(\fint_{Q_{r}(x)}|\nabla v-\nabla w|^\kappa dy +  \fint_{Q_{2r}(x)}|\nabla u-\nabla w|^\kappa dy\right)\\
&  \lesssim  \fint_{Q_{r}(x)}|\nabla u|^\kappa dy +  \left(\frac{r}{\rho}\right)^n \fint_{Q_{r}(x)}|\nabla v-\nabla w|^\kappa dy+ \left(\frac{r}{\rho}\right)^n  \fint_{Q_{2r}(x)}|\nabla u-\nabla w|^\kappa dy.
\end{align*}
Here we used 
\begin{equation}\label{decaygrad-mod}
\fint_{Q_\rho(x)}|\nabla v|^\kappa dy\lesssim  \fint_{Q_{r}(x)}|\nabla v|^\kappa dy,
\end{equation}
which is a a modified version of (2.6) in \cite[Theorem  2.1]{KM1} in the second inequality. \\
Then by  Lemma \ref{vwsigma0} we get 
\begin{align*}
\left(\fint_{Q_\rho(x)}|\nabla u|^\kappa dy\right)^{1/\kappa}  &\lesssim  \left(\fint_{Q_{r}(x)}|\nabla u|^\kappa dy\right)^{1/\ka} +   \left(\frac{r}{\rho}\right)^{n/\ka} \om(r)^{\sigma_0} \left(\fint_{Q_{2r}(x)}|\nabla w|^\kappa dy\right)^{1/\ka}\\
& \quad  + \left(\frac{r}{\rho}\right)^{n/\ka}  \left(\fint_{Q_{2r}(x)}|\nabla u-\nabla w|^\kappa dy\right)^{1/\ka}\\
& \lesssim \left\{1+ \left(\frac{r}{\rho}\right)^{n/\ka} \om(r)^{\sigma_0} \right\}\left(\fint_{Q_{2r}(x)}|\nabla u|^\kappa dy\right)^{1/\ka} \\
&   \quad +  \left\{ \left(\frac{r}{\rho}\right)^{n/\ka} \om(r)^{\sigma_0}+\left(\frac{r}{\rho}\right)^{n/\ka}\right\}  \left(\fint_{Q_{2r}(x)}|\nabla u-\nabla w|^\kappa dy\right)^{1/\ka},
\end{align*}
for a  small constant $\sigma_0>0$.
Thus  using  Lemma \ref{maininterior} and  the fact that $\om(r)\leq 2 \Lambda$, we find
\begin{align}\nonumber
\left(\fint_{Q_\rho(x)}|\nabla u|^\kappa dy\right)^{1/\kappa} 
 &\lesssim \left\{1+ \left(\frac{r}{\rho}\right)^{n/\ka} \om(r)^{\sigma_0} \right\}\left(\fint_{Q_{2r}(x)}|\nabla u|^\kappa dy\right)^{1/\ka} \nonumber\\
&+ \left(\frac{r}{\rho}\right)^{n/\ka} \left[\frac{|\mu|(Q_{3r}(x))}{r^{n-1}}\right]^{\frac{1}{p-1}}\nonumber\\
&  + \left(\frac{r}{\rho}\right)^{n/\ka} \left(\frac{|\mu|(Q_{3r}(x))}{r^{n-1}}\right) \left( \fint_{Q_{3r}(x)}|\nabla u|^\kappa dy\right)^{(2-p)/\ka}.\label{l3.2}
\end{align}
Let  $\epsilon \in (0,1)$, and choose $\rho=\epsilon r$. Then by Young's inequality we have
\begin{align*}
&\left(\fint_{Q_{\ep r}(x)}|\nabla u|^\kappa dy\right)^{1/\kappa}  \leq  C_\ep \left[\frac{|\mu|(Q_{3r}(x))}{r^{n-1}}\right]^{\frac{1}{p-1}} \\
& \qquad \qquad + \left[c_1\ep^{-n/\ka} \om(r)^{\sigma_0} + c_2\right] \left(\fint_{Q_{3r}(x)}|\nabla u|^\kappa dy\right)^{1/\ka}.
\end{align*}
Multiplying both sides by $(\epsilon r)^{1-\alpha}$, $0<\alpha\leq \bar{\alpha} <1$, and taking the supremum with respect to $r\in(0,R]$, we find 
\begin{align*}
& \sup_{0<r\leq \ep R}r^{1-\alpha}\left(\fint_{Q_{r}(x)}|\nabla u|^\kappa dy\right)^{1/\kappa}  \leq  C_\ep \sup_{0<r\leq  R} \left[\frac{|\mu|(Q_{3r}(x))}{r^{n-p+\alpha(p-1)}}\right]^{\frac{1}{p-1}} \\
&  \qquad +  \left[c_1\ep^{-n/\ka} \sup_{0<r\leq R} \om(r) + c_2\right] (\ep/3)^{1-\alpha} \sup_{0<r\leq  R} (3r)^{1-\alpha} \left(\fint_{Q_{3r}(x)}|\nabla u|^\kappa dy\right)^{1/\ka}.
\end{align*}
We now choose $\epsilon \in (0,1)$  such that
$$  c_2 (\epsilon/3)^{1-\bar{\alpha}}\leq 1/4,$$
and then choose $\bar{R}=\bar{R}(n,p,\Lambda, \bar{\alpha},\om(\cdot))>0$ and a small $\delta=\delta(n,p,\Lambda,\bar{\alpha})>0$ in \eqref{smallBMO} such that 
$$ c_1\ep ^{-n/\ka} \sup_{0<r\leq \bar{R}} \om(r)  (\ep/3)^{1-\bar{\alpha}}\leq  c_1\ep ^{-n/\ka} (2\delta)  (\ep/3)^{1-\bar{\alpha}}\leq 1/4.$$
Then it follows that 
$$\left[c_1\ep^{-n/\ka} \sup_{0<r\leq R} \om(r) + c_2\right] (\ep/3)^{1-\alpha}\leq 1/2,$$
provided $R\leq \bar{R}$. Hence, for $R\leq \bar{R}$,  we  deduce that 
\begin{align*}
& \sup_{0<r\leq \ep R}r^{1-\alpha}\left(\fint_{Q_{r}(x)}|\nabla u|^\kappa dy\right)^{1/\kappa}  \leq  C(\ep) \sup_{0<r\leq  3R} \left[\frac{|\mu|(Q_r(x))}{r^{n-p+\alpha(p-1)}}\right]^{\frac{1}{p-1}} \\
& \qquad \qquad +  \sup_{\ep R < r\leq  3R} r^{1-\alpha} \left(\fint_{Q_r(x)}|\nabla u|^\kappa dy\right)^{1/\ka}\\
& \qquad  \lesssim \left [ {\bf M}^{3R}_{p-\alpha(p-1)}(\mu)(x) \right]^{\frac{1}{p-1}} + R^{1-\alpha} \left(\fint_{Q_{3R}(x)} |\nabla u|^\kappa dy\right)^\frac{1}{\kappa}.
\end{align*}
This proves \eqref{sharpmaxugrad} in the case $R\leq \bar{R}$.  For $R>\bar{R}$, we observe that 
$${\bf M}^{\ep R}_{1-\alpha, \kappa}(\nabla u)(x)\leq {\bf M}^{\ep \bar{R}}_{1-\alpha, \kappa}(\nabla u)(x) + \left(\frac{R}{\bar{R}}\right)^{n/\ka} (\ep R)^{1-\al} \left(\fint_{Q_{\ep R}(x)} |\nabla u|^\kappa dy\right)^{1/\ka}. $$
Thus we also obtain \eqref{sharpmaxugrad} in the case $R> \bar{R}$ as long as we allow the implicit constant to depend on ${\rm diam}(\Omega)$, and $n,p,\Lambda, \bar{\alpha},\om(\cdot)$. 
\end{proof}\vspace{0.1cm}\\
In order to prove Theorem \ref{0toi1}, we need the following pointwise fractional maximal function bound.
\begin{theorem}\label{dinimax} Let $1<p\leq 2-\frac{1}{n}$ and  $u\in C^1(\Om)$ be a solution to \eqref{quasi-measure}. Suppose that $Q_{3R}(x)\subset\Om$. If for some $\sigma_1\in (0,1)$ such that $\om(\cdot)^{\sigma_1}$ is Dini-VMO, i.e., \eqref{Dini-VMO} holds, 
	then the estimate 
	\begin{align*}
	& {\bf M}^{\#,R}_{\alpha, \kappa}(u)(x)+ {\bf M}^{3R}_{1-\alpha, \kappa}(\nabla u)(x)\\
	&\qquad \lesssim  \left [ {\bf I}_{p-\alpha(p-1)}^{3R}(|\mu|)(x) \right]^{\frac{1}{p-1}} + R^{1-\alpha} \left(\fint_{Q_{3R}(x)} |\nabla u|^\kappa dy\right)^\frac{1}{\kappa}
	\end{align*}
	holds uniformly in $\alpha\in[0,1]$.	Here the implicit constant depends on $n,p,\Lambda, \bar{\alpha},\om(\cdot)$, $\sigma_1$, and ${\rm diam}(\Om)$.
\end{theorem}

\begin{proof}
	As in the proof of Theorem \ref{smallBMOTHM}, it is enough to show
	\begin{align*}
	 {\bf M}^{R}_{1-\alpha, \kappa}(\nabla u)(x) \lesssim  \left [ {\bf I}_{p-\alpha(p-1)}^{2R}(|\mu|)(x) \right]^{\frac{1}{p-1}} + R^{1-\alpha} \left(\fint_{Q_{R}(x)} |\nabla u|^\kappa dy\right)^\frac{1}{\kappa}.
	\end{align*}	
	Moreover, we may assume that $R\leq \bar{R}$, where $\bar{R}=\bar{R}(n,p,\Lambda, \sigma_1,\om(\cdot))>0$ is to be determined. \\ 
	Arguing as in the proof of \eqref{l3.2}, but this time using \eqref{oscqk'} (in Lemma \ref{osckkkfornab}) instead of \eqref{decaygrad-mod} and Lemma \ref{Dini-VMOlem} instead of Lemma  \ref{vwsigma0},
	we have for $Q_\rho(x)\subset Q_r(x) \subset Q_{3r}(x)\subset\Om$, 
	\begin{align}\label{l3.3}
	&\left(\fint_{Q_\rho(x)}|\nabla u-{\bf q}_{Q_\rho(x)}|^\kappa dy\right)^{1/\kappa} \lesssim
	  \left(\frac{\rho}{r}\right)^{\beta_0 } \left(\fint_{Q_{3r}(x)}|\nabla u -{\bf q}_{Q_{3r}(x)}|^\kappa dy\right)^{1/\ka} \nonumber\\
	&\qquad +  \left(\frac{r}{\rho}\right)^{n/\ka} \om(r)^{\sigma_1} \left(\fint_{Q_{3r}(x)}|\nabla u|^\kappa dy\right)^{1/\ka} + \left(\frac{r}{\rho}\right)^{n/\ka} \left[\frac{|\mu|(Q_{3r}(x))}{r^{n-1}}\right]^{\frac{1}{p-1}} \nonumber\\
	&\qquad +  \left(\frac{r}{\rho}\right)^{n/\ka} \left(\frac{|\mu|(Q_{3r}(x))}{r^{n-1}}\right) \left( \fint_{Q_{3r}(x)}|\nabla u|^\kappa dy\right)^{(2-p)/\ka}.
	\end{align}
	Here ${\bf q}_{Q_\rho(x)}\in\RR^n$ is defined by
	$${\bf q}_{Q_\rho(x)}:= \argmin_{{\bf q}\in\RR^n}\left(\fint_{Q_\rho(x)}|\nabla u-{\bf q}|^\kappa dy\right)^{1/\kappa}, \quad Q_\rho(x)\Subset\Om.$$
That is, ${\bf q}_{Q_\rho(x)}$ is a vector such that 
$$\inf_{{\bf q}\in\RR^n}\left(\fint_{Q_\rho(x)}|\nabla u-{\bf q}|^\kappa dy\right)^{1/\kappa}=\left(\fint_{Q_\rho(x)}|\nabla u-{\bf q}_{Q_\rho(x)}|^\kappa dy\right)^{1/\kappa}.$$
Note that for $Q_\rho(x)\subset Q_s(x) \Subset\Om$, one has
\begin{align}\label{Qronly}
|{\bf q}_{Q_s(x)}|&=\left(\fint_{Q_s(x)} |{\bf q}_{Q_s(x)}|^\ka dy\right)^{1/\ka}\nonumber\\
&\lesssim \left(\fint_{Q_s(x)} |\nabla u- {\bf q}_{Q_s(x)}|^\ka dy\right)^{1/\ka} +
\left(\fint_{Q_s(x)} |\nabla u|^\ka dy\right)^{1/\ka} \nonumber\\
&\lesssim  \left(\fint_{Q_s(x)} |\nabla u|^\ka dy\right)^{1/\ka},
\end{align}		
and also
\begin{align}\label{Qrhor}
&|{\bf q}_{Q_\rho(x)}-{\bf q}_{Q_s(x)}|=\left(\fint_{Q_\rho(x)} |{\bf q}_{Q_\rho(x)}-{\bf q}_{Q_s(x)}|^\ka dy\right)^{1/\ka}\nonumber\\
&\qquad \lesssim \left(\fint_{Q_\rho(x)}|\nabla u-{\bf q}_{Q_\rho(x)}|^\kappa dy\right)^{1/\kappa} +  \left(\fint_{Q_\rho(x)}|\nabla u-{\bf q}_{Q_s(x)}|^\kappa dy\right)^{1/\kappa}\nonumber\\
&\qquad \lesssim  \left(\frac{s}{\rho}\right)^{n/\ka} \left(\fint_{Q_s(x)}|\nabla u-{\bf q}_{Q_s(x)}|^\kappa dy\right)^{1/\kappa}.
\end{align}		
 For brevity, for any $j=0,1,2, \dots,$ and $Q_{3R}(x)\subset\Om$, we now define 
 $$Q_j=Q_{R_j}(x),\quad R_j= \ep^j R,$$
 where $ \ep\in (0,1/3)$ is to be determined, and 
 $$A_j=\left(\fint_{Q_j} |\nabla u-{\bf q}_j|^\ka dy\right)^{1/\ka}, \quad {\bf q}_j={\bf q}_{Q_j}.$$
 	Then applying \eqref{l3.3} with $\rho= \ep R_j <r=R_j/3$  we have
 \begin{align}\label{l3.3Aj}
 A_{j+1}&\leq
 c_1 \ep^{\beta_0 } A_j + c_2 \ep^{-n/\ka} \om(R_j/3)^{\sigma_1} \left(\fint_{Q_{j}}|\nabla u|^\kappa dy\right)^{1/\ka} + C_\ep  \left[\frac{|\mu|(Q_{j})}{R_j^{n-1}}\right]^{\frac{1}{p-1}} \nonumber\\
 &\quad + C_\ep \left(\frac{|\mu|(Q_{j})}{R_j^{n-1}}\right) \left( \fint_{Q_{j}}|\nabla u|^\kappa dy\right)^{(2-p)/\ka}.
 \end{align}
 By quasi-triangle inequality,  this yields 
 \begin{align*}
A_{j+1}&\leq
c_1 \ep^{\beta_0 } A_j +  c_2 \ep^{-n/\ka} \om(R_j/3)^{\sigma_1}A_j \nonumber\\
&\qquad + c_2 \ep^{-n/\ka} \om(R_j/3)^{\sigma_1}|{\bf q}_j| + C_\ep  \left[\frac{|\mu|(Q_{j})}{R_j^{n-1}}\right]^{\frac{1}{p-1}} \nonumber\\
&\qquad + C_\ep \left(\frac{|\mu|(Q_{j})}{R_j^{n-1}}\right) \left( \fint_{Q_{j}}|\nabla u|^\kappa dy\right)^{(2-p)/\ka}.
\end{align*}
 We now choose $\ep$ sufficiently small so that 	$c_1 \ep^{\beta_0 }\leq 1/4$ and then restrict $R\leq \bar{R}$, where $\bar{R}=\bar{R}(n,p,\Lambda, \sigma_1,\om(\cdot))>0$
 is such that 
 $$c_2 \ep^{-n/\ka} \sup_{0<\rho\leq \bar{R}}\om(\rho)^{\sigma_1}\leq 1/4.$$
 Then we have
 \begin{align}\label{l3.3Aj3}
 &A_{j+1}\leq
 \frac{1}{2}A_j + C \om(R_j/3)^{\sigma_1}|{\bf q}_j| + C  \left[\frac{|\mu|(Q_{j})}{R_j^{n-1}}\right]^{\frac{1}{p-1}} \nonumber\\
 &\qquad + C \left(\frac{|\mu|(Q_{j})}{R_j^{n-1}}\right) \left( \fint_{Q_{j}}|\nabla u|^\kappa dy\right)^{(2-p)/\ka}.
 \end{align}
 Summing this up over $j\in\{0, 1,, \dots, m-1\}$, $m\in \mathbb{N}$, we get
 \begin{align*}
 &\sum_{j=1}^m A_{j}\leq
 \frac{1}{2} \sum_{j=0}^{m-1} A_j + C \sum_{j=0}^{m-1} \om(R_j/3)^{\sigma_1}|{\bf q}_j| + C  \sum_{j=0}^{m-1} \left[\frac{|\mu|(Q_{j})}{R_j^{n-1}}\right]^{\frac{1}{p-1}} \nonumber\\
 &\qquad + C \sum_{j=0}^{m-1} \left(\frac{|\mu|(Q_{j})}{R_j^{n-1}}\right) \left( \fint_{Q_{j}}|\nabla u|^\kappa dy\right)^{(2-p)/\ka}.
 \end{align*}
 Hence, 
 \begin{align*}
 &\sum_{j=1}^m A_{j}\leq
 A_0 + C \sum_{j=0}^{m-1} \om(R_j/3)^{\sigma_1}|{\bf q}_j| + C  \sum_{j=0}^{m-1} \left[\frac{|\mu|(Q_{j})}{R_j^{n-1}}\right]^{\frac{1}{p-1}} \nonumber\\
 &\qquad + C \sum_{j=0}^{m-1} \left(\frac{|\mu|(Q_{j})}{R_j^{n-1}}\right) \left( \fint_{Q_{j}}|\nabla u|^\kappa dy\right)^{(2-p)/\ka}.
 \end{align*}
On the other hand, for any $m\in\mathbb{N}$, by \eqref{Qrhor} we can write 
 	\begin{align*}
 	|{\bf q}_{m+1}|= \sum_{j=0}^{m} (|{\bf q}_{j+1}| -|{\bf q}_j|) + |{\bf q}_0| \leq C \sum_{j=0}^m A_{j} + |{\bf q}_0|,
 	\end{align*}
 and therefore in view of \eqref{Qronly}, 
 \begin{align*}
 |{\bf q}_{m+1}| &\leq
  c \,A_0 + |{\bf q}_0|+ C \sum_{j=0}^{m-1} \om(R_j/3)^{\sigma_1}|{\bf q}_j| + C  \sum_{j=0}^{m-1} \left[\frac{|\mu|(Q_{j})}{R_j^{n-1}}\right]^{\frac{1}{p-1}} \nonumber\\
 &\qquad + C \sum_{j=0}^{m-1} \left(\frac{|\mu|(Q_{j})}{R_j^{n-1}}\right) \left( \fint_{Q_{j}}|\nabla u|^\kappa dy\right)^{(2-p)/\ka}\nonumber\\
 &\leq C \left(\fint_{Q_R(x)} |\nabla u|^\ka dy\right)^{1/\ka}  + C \sum_{j=0}^{m-1} \om(R_j/3)^{\sigma_1}|{\bf q}_j| + C  \sum_{j=0}^{m-1} \left[\frac{|\mu|(Q_{j})}{R_j^{n-1}}\right]^{\frac{1}{p-1}} \nonumber\\
 &\qquad + C \sum_{j=0}^{m-1} \left(\frac{|\mu|(Q_{j})}{R_j^{n-1}}\right) \left( \fint_{Q_{j}}|\nabla u|^\kappa dy\right)^{(2-p)/\ka}.
 \end{align*}
 At this point, multiplying both sides of the above inequality by $R_{m+1}^{1-\alpha}$, $m\in \mathbb{N}$, we  deduce that 
 \begin{align*}
  R_{m+1}^{1-\alpha}|{\bf q}_{m+1}|  &\lesssim  R^{1-\alpha}\left(\fint_{Q_R(x)} |\nabla u|^\ka dy\right)^{1/\ka}  + C \sum_{j=0}^{m-1} \om(R_j/3)^{\sigma_1} R_{j}^{1-\alpha} |{\bf q}_j|\\
&\quad  +  \sum_{j=0}^{m-1} \left[\frac{|\mu|(Q_{j})}{R_j^{n-p+\alpha(p-1)}}\right]^{\frac{1}{p-1}} \\
&\quad +  \sum_{j=0}^{m-1} \left(\frac{|\mu|(Q_{j})}{R_j^{n-p+\al(p-1)}}\right) 
R_j^{(1-\al)(2-p)}\left( \fint_{Q_{j}}|\nabla u|^\kappa dy\right)^{(2-p)/\ka}.
\end{align*}
Thus,
\begin{align}\label{inductboundRq}
&  R_{m+1}^{1-\alpha}|{\bf q}_{m+1}|\leq c_3 R^{1-\alpha}\left(\fint_{Q_R(x)} |\nabla u|^\ka dy\right)^{1/\ka}\nonumber \\
&\quad  + c_3 \sum_{j=0}^{m-1} \om(R_j/3)^{\sigma_1} R_{j}^{1-\alpha} |{\bf q}_j| + c_3  \left[ {\bf I}^{2R}_{p-\al(p-1)}(|\mu|)(x)\right]^{\frac{1}{p-1}}\nonumber\\
&\quad  +  c_3\, {\bf I}^{2R}_{p-\al(p-1)}(|\mu|)(x) 
\left[{\bf M}^{R}_{1-\alpha, \kappa}(\nabla u)(x)\right]^{2-p}.
 \end{align}
 We next further restrict $\bar{R}$ so that for any $R\leq \bar{R}$,
 \begin{align*}
 \sum_{j=0}^{m-1} \om(R_j/3)^{\sigma_1}\leq \frac{1}{2c_3}.
 \end{align*}
 This is possible because we have
 \begin{align}\label{sumtoint}
 \sum_{j=0}^{m-1} \om(R_j/3)^{\sigma_1}&=\om(R/3)+ \sum_{j=1}^{m-1} \om(R_j/3)^{\sigma_1}\nonumber\\
 &\leq c \int_{R/3}^{R} \om(\rho)^{\sigma_1}\frac{d\rho}{\rho} + 
 c \sum_{j=1}^{m-1} \int_{R_{j}/3}^{R_{j-1}/3} \om(\rho)^{\sigma_1} \frac{d\rho}{\rho} \nonumber\\
 &\leq c\, \int_0^{R} \om(\rho)^{\sigma_1}\frac{d\rho}{\rho},
 \end{align}
 where we used the fact that $\om(\rho_1)\leq c\, \om(\rho_2)$ provided $\rho_1\leq \rho_2\leq C \rho_1$, $C>1$.\\ 
 Then by an induction argument we deduce from \eqref{inductboundRq} that 
 \begin{align}\label{Qquant}
 &  R_{m}^{1-\alpha}|{\bf q}_{m}|\lesssim R^{1-\alpha}\left(\fint_{Q_R(x)} |\nabla u|^\ka dy\right)^{1/\ka} \nonumber \\
 & \quad +   \left[ {\bf I}^{2R}_{p-\al(p-1)}(|\mu|)(x)\right]^{\frac{1}{p-1}}  +   {\bf I}^{2R}_{p-\al(p-1)}(|\mu|)(x) 
 \left[{\bf M}^{R}_{1-\alpha, \kappa}(\nabla u)(x)\right]^{2-p},
 \end{align} 
for every integer $m\geq 0$. \\
Let us call the right-hand side of \eqref{Qquant} by ${\bf Q}$. Then from \eqref{l3.3Aj3} and simple manipulations we obtain
\begin{align*}
A_{m+1}\leq \frac{1}{2} A_m + c|{\bf q}_m| + c R_m^{\al-1}{\bf Q},
\end{align*}
which by \eqref{Qquant}  yields
\begin{align*}
R_{m+1}^{1-\al} A_{m+1} \leq \frac{1}{2} R_{m}^{1-\al} A_m + c R_{m}^{1-\al} |{\bf q}_m| + c {\bf Q}\leq \frac{1}{2} R_{m}^{1-\al} A_m +  c {\bf Q}.
\end{align*}
As $R_{0}^{1-\al} A_{0}\leq  c {\bf Q}$, by iteration we get
\begin{align}\label{RmAm}
R_{m}^{1-\al} A_{m} \leq C  {\bf Q},
\end{align}
for every integer $m\geq 0$.\\
To conclude the proof, we observe that 
\begin{align*}
{\bf M}^{R}_{1-\alpha, \kappa}(\nabla u)(x) &\leq C\sup_{m\geq 0} R_m^{1-\alpha} \left(\fint_{Q_m} |\nabla u|^\ka dy \right)^{1/\ka}\\
&\leq   C [R_{m}^{1-\al} A_{m} +  R_{m}^{1-\al} {\bf q}_{m}]\leq   C {\bf Q},
\end{align*}
where we used \eqref{Qquant} and \eqref{RmAm} in the last inequality. Then recalling the definition of ${\bf Q}$ and using Young's inequality we obtain
\begin{align*}
{\bf M}^{R}_{1-\alpha, \kappa}(\nabla u)(x) &\leq    C R^{1-\alpha} \left(\fint_{Q_{R}(x)} |\nabla u|^\kappa dy\right)^\frac{1}{\kappa} +C \left [ {\bf I}_{p-\alpha(p-1)}^{2R}(|\mu|)(x) \right]^{\frac{1}{p-1}}\\
&\qquad +\frac{1}{2}{\bf M}^{R}_{1-\alpha, \kappa}(\nabla u)(x).
\end{align*}
This completes the proof of the theorem.
\end{proof}\vspace{0.15cm}\\
The following pointwise sharp fractional maximal function bound will be used in the proof of Theorem  \ref{gradholdTHM}.

\begin{theorem}\label{supdnvthm} Let $1<p\leq 2-\frac{1}{n}$ and  $u\in C^1(\Om)$ be a solution to \eqref{quasi-measure}. Suppose that $Q_{3R}(x)\subset\Om$. If for some $\sigma_1\in (0,1)$ such that  
	\begin{equation}\label{Dini-VMO-stronger}
	\sup_{0<\rho\leq 1} \frac{\om(\rho)^{\sigma_1}}{\rho^{\bar{\alpha}}}\leq K,
	\end{equation}
	for some $\bar{\alpha}\in [0,\beta_0)$, 
		then the estimate 
	\begin{align*}
	 {\bf M}^{\#, 3R}_{\alpha, \kappa}(\nabla u)(x) & \lesssim  \left [ {\bf M}^{3R}_{1-\alpha, \kappa}(\mu)(x) \right]^{\frac{1}{p-1}}\\
	&\qquad +  \left [ {\bf I}_{1}^{3R}(|\mu|)(x) \right]^{\frac{1}{p-1}} + R^{-\alpha} \left(\fint_{Q_{3R}(x)} |\nabla u|^\kappa dy\right)^\frac{1}{\kappa}
	\end{align*}
	holds uniformly in $\alpha\in[0,\bar{\alpha}]$. Here $\beta_0$ is as in Lemma \ref{osckkkfornab}, and the implicit constant depends on $n,p,\Lambda, \bar{\alpha},\om(\cdot), \sigma_1, K$, and ${\rm diam}(\Om)$.
\end{theorem}
\begin{remark}
	Condition \eqref{Dini-VMO-stronger} implies the Dini-VMO condition   \eqref{Dini-VMO}. In turns, \eqref{Dini-VMO} implies \eqref{smallBMO}, whereas 
\eqref{Dini-VMO-stronger} is implied by	the Dini-H\"older condition \eqref{dhcond}.
\end{remark}
\begin{proof}
	It suffices to show
\begin{align*}
 {\bf M}^{\#, R}_{\alpha, \kappa}(\nabla u)(x) & \lesssim  \left [ {\bf M}^{R}_{1-\alpha, \kappa}(\mu)(x) \right]^{\frac{1}{p-1}}\\
&\qquad +  \left [ {\bf I}_{1}^{R}(|\mu|)(x) \right]^{\frac{1}{p-1}} + R^{-\alpha} \left(\fint_{Q_{R}(x)} |\nabla u|^\kappa dy\right)^\frac{1}{\kappa},
\end{align*}	
for $R\leq 1$, where the implicit constant depends on $n,p,\Lambda, \bar{\alpha},\om(\cdot), \sigma_1, K$, and ${\rm diam}(\Om)$.\\
With the notation used in proof of Theorem \ref{dinimax},  multiplying both sides of \eqref{l3.3Aj} by   $R_{j+1}^{-\al}$, $j\geq 0$, we have 
\begin{align*}
R_{j+1}^{-\al} A_{j+1}&\leq
c_1 \ep^{\beta_0 -\alpha} R_{j}^{-\al} A_j \nonumber\\
&\quad + C_\ep R_{j}^{-\al}\om(R_j/3)^{\sigma_1} \left(\fint_{Q_{j}}|\nabla u|^\kappa dy\right)^{1/\ka} + C_\ep R_{j}^{-\al} \left[\frac{|\mu|(Q_{j})}{R_j^{n-1}}\right]^{\frac{1}{p-1}} \nonumber\\
&\quad + C_\ep  R_{j}^{-\al}\left(\frac{|\mu|(Q_{j})}{R_j^{n-1}}\right) \left( \fint_{Q_{j}}|\nabla u|^\kappa dy\right)^{(2-p)/\ka}.
\end{align*}
This time we  choose $\ep\in (0,1/3)$ such  that 
$$c_1 \ep^{\beta_0 -\alpha}\leq c_1 \ep^{\beta_0 -\bar{\alpha}}\leq \frac{1}{2},$$
and employ \eqref{Dini-VMO-stronger} together  with  the restriction $R_j\leq 1$, to deduce
\begin{align}\label{l3.3Aj-til}
R_{j+1}^{-\al} A_{j+1}&\leq \frac{1}{2}  R_{j}^{-\al} A_j  + C K \left(\fint_{Q_{j}}|\nabla u|^\kappa dy\right)^{1/\ka} + C   \left[\frac{|\mu|(Q_{j})}{R_j^{n-1+\al}}\right]^{\frac{1}{p-1}} \nonumber\\
&\quad + C \left(\frac{|\mu|(Q_{j})}{R_j^{n-1 +\al}}\right) \left( \fint_{Q_{j}}|\nabla u|^\kappa dy\right)^{(2-p)/\ka}.
\end{align}
On the other hand, applying Theorem \ref{dinimax} in the case $\al=1$, we can bound
$$\left( \fint_{Q_{j}}|\nabla u|^\kappa dy\right)^{1/\ka}\lesssim \left [ {\bf I}_{1}^{R}(|\mu|)(x) \right]^{\frac{1}{p-1}} +  \left(\fint_{Q_{R}(x)} |\nabla u|^\kappa dy\right)^\frac{1}{\kappa}$$
for every integer $j\geq 0$. Thus, using \eqref{l3.3Aj-til} and Young's inequality we get 
\begin{align*}
R_{j+1}^{-\al} A_{j+1}&\leq \frac{1}{2}  R_{j}^{-\al} A_j  + C \left[{\bf M}^{R}_{1-\al, \ka}(\mu)(x)\right]^{\frac{1}{p-1}} \nonumber\\
&\quad + C \left [ {\bf I}_{1}^{R}(|\mu|)(x) \right]^{\frac{1}{p-1}} + C  \left(\fint_{Q_{R}(x)} |\nabla u|^\kappa dy\right)^\frac{1}{\kappa}.
\end{align*}
 Iterating this inequality, we find for any  $m\in \mathbb{N}$,
\begin{align*}
R_{m}^{-\al} A_{m}&\leq 2^{-m}  R_{0}^{-\al} A_0  + C \left[{\bf M}^{R}_{1-\al, \ka}(\mu)(x)\right]^{\frac{1}{p-1}}\nonumber\\
&\quad + C \left [ {\bf I}_{1}^{R}(|\mu|)(x) \right]^{\frac{1}{p-1}} + C  \left(\fint_{Q_{R}(x)} |\nabla u|^\kappa dy\right)^\frac{1}{\kappa}\\
&\leq  C \left[{\bf M}^{R}_{1-\al, \ka}(\mu)(x)\right]^{\frac{1}{p-1}}\nonumber\\
&\quad + C \left [ {\bf I}_{1}^{R}(|\mu|)(x) \right]^{\frac{1}{p-1}} + C R^{-\al} \left(\fint_{Q_{R}(x)} |\nabla u|^\kappa dy\right)^\frac{1}{\kappa}.
\end{align*}
In view  of the fact that 
$${\bf M}^{\#, R}_{\alpha, \kappa}(\nabla u)(x)\lesssim \sup_{m\geq 0} R_{m}^{-\al} A_{m},$$
this completes the proof of the theorem. 
\end{proof}

\section{Proof of Theorem \ref{uptoalpha0}}

\begin{proof}[Proof of Theorem \ref{uptoalpha0}]  For any cube  $Q_\rho(x)\Subset\Om$, let $q_{Q_\rho(x)}\in\RR$ be defined by
	$$q_{Q_\rho(x)}:= \argmin_{ q\in\RR}\left(\fint_{Q_\rho(x)}| u- q|^\kappa dy\right)^{1/\kappa},$$
i.e.,  $q_{Q_\rho(x)}$ is a real number such that 
$$\inf_{q\in\RR^n}\left(\fint_{Q_\rho(x)}|u-q|^\kappa dy\right)^{1/\kappa}=\left(\fint_{Q_\rho(x)}|u-q_{Q_\rho(x)}|^\kappa dy\right)^{1/\kappa}.$$	
Then using  quasi-triangle inequality  a few times and Lemma \ref{osckkk},  we have for $Q_\rho(x)\subset Q_{r}(x) \subset Q_{3r}(x) \subset\Om$,
\begin{align*}
\left(\fint_{Q_\rho(x)}| u- q_{Q_{\rho}(x)}|^\kappa dy\right)^{1/\kappa}	
&\lesssim 	\left(\frac{\rho}{r}\right)^{\al_0} \left(\fint_{Q_{2r}(x)} |u-q_{Q_{2r}(x)}|^\kappa dy\right)^{\frac{1}{\kappa}}\\
&\qquad + \left(\frac{\rho}{r}\right)^{-n/\kappa}	\left(\fint_{Q_{ 2r}(x_0)} |u-w|^\kappa dy\right)^{\frac{1}{\kappa}}.	
\end{align*}	
Here we choose $w$ as in 	\eqref{eq1} with $Q_{2r}(x)$ in place of $Q_{2R}(x_0)$.	\\
We now apply  Remark \ref{RemarkCOM} to bound the second term on the right-hand side of the above inequality. This yields that 
\begin{align*}
\left(\fint_{Q_\rho(x)}| u- q_{Q_{\rho}(x)}|^\kappa dy\right)^{1/\kappa}	
&\lesssim 	\left(\frac{\rho}{r}\right)^{\al_0} \left(\fint_{Q_{2r}(x)} |u-q_{Q_{2r}(x)}|^\kappa dy\right)^{\frac{1}{\kappa}}\\
&\quad + \left(\frac{\rho}{r}\right)^{-n/\kappa} \left(\frac{|\mu|(Q_{3r}(x))}{r^{n-p}}\right)^{\frac{1}{p-1}}\\
&\quad  + \left(\frac{\rho}{r}\right)^{-n/\kappa} \frac{|\mu|(Q_{3r}(x))}{r^{n-p}} \left(	\fint_{Q_{3r}(x)}|u-q_{Q_{3r}(x)}|^{\kappa} dy\right)^{\frac{2-p}{\kappa}}.
\end{align*}
Letting $\rho=\ep r$, $\ep \in (0,1)$, and using Young's inequality we find
\begin{align}\label{intA}
\left(\fint_{Q_{\ep r}(x)}| u- q_{Q_{\ep r}(x)}|^\kappa dy\right)^{1/\kappa}	
&\lesssim 	C \ep^{\al_0} \left(\fint_{Q_{3r}(x)} |u-q_{Q_{3r}(x)}|^\kappa dy\right)^{\frac{1}{\kappa}}\nonumber\\
&\qquad +C_\epsilon\left(\frac{|\mu|(Q_{3r}(x_0))}{r^{n-p}}\right)^{\frac{1}{p-1}}.
\end{align}	
Next,	we  choose $\ep \in (0,1/3)$ small enough so that $
		C\epsilon^{\alpha_0}\leq \frac{1}{2}$, where $C$ is the constant in \eqref{intA}. Let $Q_{R}(x_0)\subset\Om$ be as given in the theorem. Then for any cube $Q_\delta(x)\subset Q_{R}(x_0)$  we	set $\delta_j=\epsilon^{j} \delta$, $Q_j=Q_{\delta_j}(x)$, $q_j=q_{Q_j}$, $j\geq 0$, and define
		$$B_j:= \left(\fint_{Q_{j}}| u- q_{Q_j}|^\kappa dy\right)^{1/\kappa}.$$
		  Applying \eqref{intA}  with $r=\delta_j/3$ yields 
		$$
		B_{j+1} \leq \frac{1}{2}B_j +C\left( \frac{|\mu|(Q_{j})}{\delta_{j}^{n-p}}\right)^{\frac{1}{p-1}}. 
		$$
		Summing this up over $j\in \{1,3,...,m-1\}$, we obtain 
		\begin{align*}
		\sum_{j=1}^{m} B_j &\leq C\, B_1 +C\sum_{j=1}^{m-1}\left(\frac{|\mu|(Q_{j})}{\delta_{j}^{n-p}}\right)^{\frac{1}{p-1}}.
		\end{align*}
		As in \eqref{Qrhor}, we have  
		$$|q_{j+1}-q_{j}| \leq C B_j $$
		for all integers $j\geq 1$, and thus 
		\begin{align}\label{qwolff}
		|q_m|  &\leq |q_m-q_1|+q_1  \leq q_1+C \sum_{j=1}^{m-1} B_j\nonumber\\
		&\leq q_1 +C B_1 + C \sum_{j=1}^{m-1}\left(\frac{|\mu|(Q_{j})}{\delta_{j}^{n-p}}\right)^{\frac{1}{p-1}}\nonumber\\
		&\leq C \left(\fint_{Q_{1}} |u|^\kappa dx\right)^{\frac{1}{\kappa}} + C \sum_{j=1}^{m-1}\left(\frac{|\mu|(Q_{j})}{\delta_{j}^{n-p}}\right)^{\frac{1}{p-1}}
		\end{align}
		holds for every integer $m\geq 2$. Here we use the simple fact (see \eqref{Qronly}) that  
		\begin{align*}
		B_1 +q_1 &\leq   C \left(\fint_{Q_{1}} |u|^\kappa dx\right)^{\frac{1}{\kappa}}.
		\end{align*}
Now for $x,y\in Q_{R/8}(x_0)$ we choose $$\delta=\frac{1}{2}|x-y|_\infty=\frac{1}{2}\max_{1\leq i\leq n}|x_i-y_i|.$$   Note that $\delta < R/8$ and $Q_{\delta}(y)\subset Q_{3\delta}(x)\subset Q_{R/2}(x_0)$. Then applying   \eqref{qwolff}, we have 
\begin{align*}
|q_m| \leq C  \left(\fint_{Q_{\delta}(x)} |u|^\kappa dz\right)^{\frac{1}{\kappa}} + C  \delta^\alpha\sum_{j=1}^{m-1}\left(\frac{|\mu|(Q_{\delta_j}(x))}{\delta_{j}^{n-p+\alpha(p-1)}}\right)^{\frac{1}{p-1}}.
\end{align*}
Sending $m\rightarrow\infty$ and using \cite[Lemma 4.1]{DS}, we get
\begin{align*}
|u(x)| \leq C  \left(\fint_{Q_\delta(x)} |u|^\kappa dz\right)^{\frac{1}{\kappa}} + C \delta^\alpha {\bf W}_{1-\alpha(p-1)/p,p}^R(|\mu|)(x).
\end{align*}
Since $u-m, m\in\RR,$ is  also a solution of \eqref{quasi-measure}, it follows that 
\begin{align*}
|u(x)-m| &\leq C  \left(\fint_{Q_{\delta}(x)} |u-m|^\kappa dz\right)^{\frac{1}{\kappa}} + C \delta^\alpha {\bf W}_{1-\alpha(p-1)/p,p}^R(|\mu|)(x)\\
&\leq C  \left(\fint_{Q_{3\delta}(x)} |u-m|^\kappa dz\right)^{\frac{1}{\kappa}} + C \delta^\alpha {\bf W}_{1-\alpha(p-1)/p,p}^R(|\mu|)(x).
\end{align*}
Likewise, we have
\begin{align*}
|u(y)-m| &\leq C  \left(\fint_{Q_{\delta}(y)} |u-m|^\kappa dz\right)^{\frac{1}{\kappa}} + C \delta^\alpha {\bf W}_{1-\alpha(p-1)/p,p}^R(|\mu|)(y)\\
&\leq C  \left(\fint_{Q_{3\delta}(x)} |u-m|^\kappa dz\right)^{\frac{1}{\kappa}} + C \delta^\alpha {\bf W}_{1-\alpha(p-1)/p,p}^R(|\mu|)(y).
\end{align*}
Now choosing $m=q_{Q_{3\delta}(x)}$ we find
\begin{align}\label{diffxy}
|u(x)-u(y)| &\leq C  \left(\fint_{Q_{3\delta}(x)} |u-q_{Q_{3\delta}(x)}|^\kappa dz\right)^{\frac{1}{\kappa}}\nonumber\\
&\quad  + C \delta^\alpha \Big[ {\bf W}_{1-\alpha(p-1)/p,p}^R(|\mu|)(x)+ {\bf W}_{1-\alpha(p-1)/p,p}^R(|\mu|)(y)\Big].
\end{align}
On the other hand, by Theorem \ref{sharpmax} and the fact that $3\delta< 3R/8,$ we have
\begin{align}\label{CCOO}
&\left(\fint_{Q_{3\delta}(x)} |u-q_{Q_{3\delta}(x)}|^\kappa dz\right)^{\frac{1}{\kappa}}\nonumber\\
&\qquad \lesssim  \delta^{\alpha} \left [ {\bf M}^{9R/16}_{p-\alpha(p-1)}(\mu)(x) \right]^{\frac{1}{p-1}} + \left(\frac{\delta}{R}\right)^{\alpha} R\left(\fint_{Q_{9R/16}(x)} |\nabla u|^\kappa dz\right)^\frac{1}{\kappa} \nonumber\\
&\qquad \lesssim \delta^\alpha  {\bf W}_{1-\alpha(p-1)/p,p}^R(|\mu|)(x)
+\left(\frac{\delta}{R}\right)^\alpha \left(\fint_{B_{R}(x)} |u|^\kappa dz\right)^{\frac{1}{\kappa}},
\end{align}
where we used a Caccioppoli type inequality of \cite[Corollary 2.4]{HP3} in the last bound.\\
Combining inequalities \eqref{diffxy} and \eqref{CCOO}, we complete the proof of the theorem.
\end{proof}

\section{Proof of Theorems \ref{upto1} and \ref{0toi1}}
\begin{proof}[Proof of Theorems \ref{upto1}]
The main idea of the proof of Theorem \ref{upto1} lies in the proof of \cite[Theorem 1.2]{KM1}. First, in view  of Theorem  
\ref{uptoalpha0}, it suffices to prove \eqref{fracalphau-new} uniformly in $\al\in [\al_0/2, \bar{\alpha}]$, $\bar{\alpha}<1$, for all $x, y\in Q_{R/8}(x_0)$.  
\\
On the other hand, for a.e. $x, y\in Q_{R/8}(x_0)$ and $f\in L^\ka(Q_R(x_0))$,  we have 
 the  inequality 
\begin{align*}
|f(x)-f(y)|\leq \left( \frac{c}{\al} \right) |x-y|^\alpha & 
 \left[ {\bf M}^{\#, R/2}_{\al, \ka}(f)(x) + {\bf M}^{\#, R/2}_{\al, \ka}(f)(y) \right],
\end{align*}
 provided $\alpha\in (0,1]$. See inequalities (4.9) and (4.10) in 
\cite{DS}.
Applying this with $f=u$ and $\al\in [\al_0/2, \bar{\alpha}]$, and using Theorem \ref{smallBMOTHM}, we obtain 
\begin{align*}
& |u(x)-u(y)|\leq \left( \frac{c}{\al_0} \right) |x-y|^\alpha  \left[ {\bf M}^{3R/4}_{p-\alpha(p-1)}(\mu)(x) + {\bf M}^{3R/4}_{p-\alpha(p-1)}(\mu)(y)\right]^{\frac{1}{p-1}}\\
 & + \left( \frac{c}{\al_0} \right) |x-y|^\alpha  R^{1-\alpha} \left\{\left(\fint_{Q_{3R/4}(x)} |\nabla u|^\kappa dy\right)^\frac{1}{\kappa} +\left(\fint_{Q_{3R/4}(y)} |\nabla u|^\kappa dy\right)^\frac{1}{\kappa}\right\}.
\end{align*}
Then invoking the Caccioppoli type inequality of \cite[Corollary 2.4]{HP3} we obtain \eqref{fracalphau-new} uniformly in $\al\in [\al_0/2, \bar{\alpha}]$ as desired.
\end{proof}\vspace{0.2cm}\\
\begin{proof}[Proof of Theorem \ref{0toi1}]
The proof of Theorem \ref{0toi1} is similar to that of Theorems \ref{upto1}, but this time we use Theorem \ref{dinimax}  instead of Theorem   \ref{smallBMOTHM}.
\end{proof}

\section{Proof of Theorem \ref{gradholdTHM}}
\begin{proof}[Proof of Theorem \ref{gradholdTHM}]
	Let $Q_R(x_0)\subset\Om$ be as given in the theorem. For any $x,y\in Q_{R/4}(x_0)$, we set $\delta=\frac{1}{2}|x-y|_\infty$. Note that  $\de<R/4$
	and $Q_\delta(y)\subset Q_{3\delta}(x)\subset Q_{R}(x_0)$. We shall keep the notation in the proof of Theorem \ref{dinimax} except that  we replace $R$ with $\de$  so that
	$R_j=\de_j=\ep^j \de$, $Q_j=Q_{\ep^j \de}(x)$,  ${\bf q}_j= {\bf q}_{Q_{\ep^j \de}(x)}$, and 
	$$A_j=\left(\fint_{Q_{\ep^j \de}(x)} |\nabla u-{\bf q}_{Q_{\ep^j \de}(x)}|^\ka dy\right)^{1/\ka}=\left(\fint_{Q_{j}} |\nabla u-{\bf q}_{j}|^\ka dy\right)^{1/\ka}$$
	for all integers $j\geq 0$.\\
	Then by choosing $\ep$ in \eqref{l3.3Aj} such  that $c_1 \ep^{\beta_0 }\leq 1/2$, we have 
	\begin{align*}
	A_{j+1}&\leq
	\frac{1}{2} A_j +  C_\ep \om(\de_j/3)^{\sigma_1} \left(\fint_{Q_{j}}|\nabla u|^\kappa dy\right)^{1/\ka} + C_\ep  \left[\frac{|\mu|(Q_{j})}{\de_j^{n-1}}\right]^{\frac{1}{p-1}} \nonumber\\
	&\quad + C_\ep \left(\frac{|\mu|(Q_{j})}{\de_j^{n-1}}\right) \left( \fint_{Q_{j}}|\nabla u|^\kappa dy\right)^{(2-p)/\ka}.
	\end{align*}
	 Summing this up over $j\in\{0, 1,, \dots, m-1\}$, $m\in \mathbb{N}$, and then simplifying,   we get
	 \begin{align*}
	 \sum_{j=1}^m A_{j}&\leq
	 A_0 + C \sum_{j=0}^{m-1} \om(\de_j/3)^{\sigma_1} \left(\fint_{Q_{j}}|\nabla u|^\kappa dy\right)^{1/\ka} + C  \sum_{j=0}^{m-1} \left[\frac{|\mu|(Q_{j})}{\de_j^{n-1}}\right]^{\frac{1}{p-1}} \nonumber\\
	 &\qquad + C \sum_{j=0}^{m-1} \left(\frac{|\mu|(Q_{j})}{\de_j^{n-1}}\right) \left( \fint_{Q_{j}}|\nabla u|^\kappa dy\right)^{(2-p)/\ka}.
	 \end{align*}
	 On the other hand,  by  \eqref{Qrhor},
	 \begin{align*}
	 |{\bf q}_{m+1}-{\bf m}| &=  \sum_{j=0}^{m} (|{\bf q}_{j+1}-{\bf m}| -|{\bf q}_j-{\bf m}|) + |{\bf q}_0-{\bf m}|\\
	 &\leq \sum_{j=0}^{m} (|{\bf q}_{j+1} -{\bf q}_j| + |{\bf q}_0-{\bf m}|
	  \leq C \sum_{j=0}^m A_{j} + |{\bf q}_0-{\bf m}|,
	 \end{align*}
	 which holds for any ${\bf m}\in\RR^n$ and integer $m\geq 0$.\\
	 Hence, it follows that 
	 \begin{align*}
	  |{\bf q}_{m+1}-{\bf m}| &\leq
	 C \sum_{j=0}^{m-1} \om(\de_j/3)^{\sigma_1} \left(\fint_{Q_{j}}|\nabla u|^\kappa dy\right)^{1/\ka} + C  \sum_{j=0}^{m-1} \left[\frac{|\mu|(Q_{j})}{\de_j^{n-1}}\right]^{\frac{1}{p-1}} \nonumber\\
	 &\quad + C \sum_{j=0}^{m-1} \left(\frac{|\mu|(Q_{j})}{\de_j^{n-1}}\right) \left( \fint_{Q_{j}}|\nabla u|^\kappa dy\right)^{(2-p)/\ka} +
	  C\, A_0 + |{\bf q}_0-{\bf m}|.
	 \end{align*}
	 Then  using
	 $$|{\bf q}_0-{\bf m}|\lesssim \left(\fint_{Q_{\de}(x)}|\nabla u-{\bf m}|^\kappa dy\right)^{1/\ka},$$
	 which can be proved as in \eqref{Qronly}, we get 
	 \begin{align*}
	 |{\bf q}_{m+1}-{\bf m}| 	 &\lesssim
	  \sum_{j=0}^{m-1} \om(\de_j/3)^{\sigma_1} \left(\fint_{Q_{j}}|\nabla u|^\kappa dy\right)^{1/\ka} +   \sum_{j=0}^{m-1} \left[\frac{|\mu|(Q_{j})}{\de_j^{n-1}}\right]^{\frac{1}{p-1}} \nonumber\\
	 &\quad +  \sum_{j=0}^{m-1} \left(\frac{|\mu|(Q_{j})}{\de_j^{n-1}}\right) \left( \fint_{Q_{j}}|\nabla u|^\kappa dy\right)^{(2-p)/\ka}\\
	 &\quad  + \left(\fint_{Q_{\de}(x)}|\nabla u-{\bf q}_{Q_\de(x)}|^\kappa dy\right)^{1/\ka} + \left(\fint_{Q_{\de}(x)}|\nabla u-{\bf m}|^\kappa dy\right)^{1/\ka}.
	 \end{align*}
We next set $$M(x,r):= \left[{\bf I}_1^{r}(|\mu|)(x)\right]^{\frac{1}{p-1}}+ \left(\fint_{Q_r(x)} |\nabla u|^\ka dz\right)^{1/\ka}, \quad r>0.$$	
Then  applying Theorem 
\ref{dinimax} with $\al=1$ and $3R=\de$, we have
$$\left(\fint_{Q_{j}}|\nabla u|^\kappa dy\right)^{1/\ka}\lesssim M(x,\de),\qquad \forall j\geq 0.$$	 
Plugging this into the last bound for 	 $|{\bf q}_{m+1}-{\bf m}|$ we deduce that 
\begin{align*}
|{\bf q}_{m+1}-{\bf m}| &\lesssim
 \de^\al \sum_{j=0}^{m-1} \de_j^{-\al} \om(\de_j/3)^{\sigma_1}M(x,\de) + \de^\al  \sum_{j=0}^{m-1} \left[\frac{|\mu|(Q_{j})}{\de_j^{n-1 +\al}}\right]^{\frac{1}{p-1}} \de^{\frac{\al(2-p)}{p-1}} \nonumber\\
&\quad +  \de^\al \sum_{j=0}^{m-1} \left(\frac{|\mu|(Q_{j})}{\de_j^{n-1 +\al}}\right)  M(x,\de)^{2-p}\\
&\quad  + \left(\fint_{Q_{\de}(x)}|\nabla u-{\bf q}_{Q_\de(x)}|^\kappa dy\right)^{1/\ka} + \left(\fint_{Q_{\de}(x)}|\nabla u-{\bf m}|^\kappa dy\right)^{1/\ka}.
\end{align*}
Also, note that as in \eqref{sumtoint} we have 
$$\sum_{j=0}^{m-1} \de_j^{-\al} \om(\de_j/3)^{\sigma_1}\lesssim \sum_{j=0}^{m-1} \de_j^{-\bar{\al}} \om(\de_j/3)^{\sigma_1}\lesssim \int_{0}^\de \frac{\om(\rho)^{\sigma_1}}{\rho^{\bar{\al}}}\frac{d\rho}{\rho}\lesssim \int_{0}^{R/4} \frac{\om(\rho)^{\sigma_1}}{\rho^{\bar{\al}}}\frac{d\rho}{\rho}.$$
At this point, using the Dini-H\"older condition \eqref{dhcond}, we obtain, after some simple manipulations,  	 
\begin{align*}
|{\bf q}_{m+1}-{\bf m}| &\lesssim
\de^\al M(x,\de) + \de^\al   \left[ {\bf I}_{1-\al}^{2\de}(|\mu|)(x)\right]^{\frac{1}{p-1}}  +  \de^\al {\bf I}_{1-\al}^{2\de}(|\mu|)(x)  M(x,\de)^{2-p}\\
&\quad  + \left(\fint_{Q_{\de}(x)}|\nabla u-{\bf q}_{Q_\de(x)}|^\kappa dy\right)^{1/\ka} + \left(\fint_{Q_{\de}(x)}|\nabla u-{\bf m}|^\kappa dy\right)^{1/\ka}.
\end{align*}
Here we also used  that $\de< R/4 < {\rm diam}(\Om)$ and the implicit constants are allowed to depend on ${\rm diam}(\Om)$.\\
Thus letting $m\rightarrow \infty$ and using Young's inequality we obtain
\begin{align*}
|\nabla u(x)-{\bf m}| &\lesssim
\de^\al M(x,\de) + \de^\al   \left[ {\bf I}_{1-\al}^{R}(|\mu|)(x)\right]^{\frac{1}{p-1}}  \\
&\quad  + \left(\fint_{Q_{\de}(x)}|\nabla u-{\bf q}_{Q_\de(x)}|^\kappa dz\right)^{1/\ka} + \left(\fint_{Q_{\de}(x)}|\nabla u-{\bf m}|^\kappa dz\right)^{1/\ka}.
\end{align*}
Likewise, we also have 
\begin{align*}
|\nabla u(y)-{\bf m}| &\lesssim
\de^\al M(y,\de) + \de^\al   \left[ {\bf I}_{1-\al}^{R}(|\mu|)(y)\right]^{\frac{1}{p-1}}  \\
&\quad  + \left(\fint_{Q_{\de}(y)}|\nabla u-{\bf q}_{Q_\de(y)}|^\kappa dz\right)^{1/\ka} + \left(\fint_{Q_{\de}(y)}|\nabla u-{\bf m}|^\kappa dz\right)^{1/\ka}\\
&\lesssim
\de^\al M(y,\de) + \de^\al   \left[ {\bf I}_{1-\al}^{R}(|\mu|)(y)\right]^{\frac{1}{p-1}}  \\
&\quad  + \left(\fint_{Q_{3\de}(x)}|\nabla u-{\bf q}_{Q_{3\de}(x)}|^\kappa dz\right)^{1/\ka} + \left(\fint_{Q_{3\de}(x)}|\nabla u-{\bf m}|^\kappa dz\right)^{1/\ka},
\end{align*}
where we used that $Q_\de(y)\subset Q_{3\de}(x)$.\\
Combining these two estimates and choosing ${\bf m}={\bf q}_{Q_{3\de}(x)}$, we find
\begin{align}\label{nabu-xy}
& |\nabla u(x)-\nabla u(y)| \lesssim
 \de^\al  \left\{ \left[ {\bf I}_{1-\al}^{R}(|\mu|)(x)\right]^{\frac{1}{p-1}} + \left[ {\bf I}_{1-\al}^{R}(|\mu|)(y)\right]^{\frac{1}{p-1}}\right\} \nonumber\\
&\qquad + \de^\al \left[M(x,\de) + M(y,\de)\right]   + \left(\fint_{Q_{3\de}(x)}|\nabla u-{\bf q}_{Q_{3\de}(x)}|^\kappa dz\right)^{1/\ka}.
\end{align}
As $\de< R/4$ and $Q_{R/4}(x)\cup Q_{R/4}(y) \subset Q_{R}(x_0)$, we can apply Theorem  \ref{dinimax} with $\al=1$ to have the bound
\begin{align}\label{M-xy}
M(x,\de) + M(y,\de)&\lesssim \left[ {\bf I}_{1}^{R/4}(|\mu|)(x)\right]^{\frac{1}{p-1}} + \left(\fint_{Q_{R/4}(x)}|\nabla u|^\kappa dz\right)^{1/\ka}\nonumber\\
&\qquad +\left[ {\bf I}_{1}^{R/4}(|\mu|)(y)\right]^{\frac{1}{p-1}} + \left(\fint_{Q_{R/4}(y)}|\nabla u|^\kappa dz\right)^{1/\ka}\nonumber\\
&\lesssim \left[ {\bf I}_{1-\al}^{R}(|\mu|)(x)\right]^{\frac{1}{p-1}} + \left[ {\bf I}_{1-\al}^{R}(|\mu|)(y)\right]^{\frac{1}{p-1}}\nonumber\\
&\qquad +  R^{-\al} \left(\fint_{Q_{R}(x_0)}|\nabla u|^\kappa dz\right)^{1/\ka}.
\end{align}
Similarly, we can use Theorem \ref{supdnvthm} to bound the last term on the right-hand of \eqref{nabu-xy} as follows:
\begin{align}\label{oscq3de}
 &\left(\fint_{Q_{3\de}(x)}|\nabla u-{\bf q}_{Q_{3\de}(x)}|^\kappa dz\right)^{1/\ka}\lesssim \de^\al {\bf M}^{\#, 3R/4}_{\al, \ka}(\nabla u)(x)\nonumber\\
 & \qquad \lesssim \de^\al \left [ {\bf M}^{3R/4}_{1-\alpha, \kappa}(\mu)(x) \right]^{\frac{1}{p-1}} +  \de^\al \left [ {\bf I}_{1}^{3R/4}(|\mu|)(x) \right]^{\frac{1}{p-1}} \nonumber\\
 & \qquad \qquad +  \left(\frac{\de}{R}\right)^{\alpha} \left(\fint_{Q_{3R/4}(x)} |\nabla u|^\kappa dz\right)^\frac{1}{\kappa}\nonumber\\
 &\qquad \lesssim \de^\al  \left [ {\bf I}_{1-\al}^{R}(|\mu|)(x) \right]^{\frac{1}{p-1}}  +  \left(\frac{\de}{R}\right)^{\alpha} \left(\fint_{Q_{R}(x_0)} |\nabla u|^\kappa dz\right)^\frac{1}{\kappa}.
\end{align}
We now plug estimates \eqref{M-xy} and \eqref{oscq3de} into \eqref{nabu-xy} to arrive at 
\begin{align*}
 |\nabla u(x)-\nabla u(y)| &\lesssim
\de^\al  \left\{ \left[ {\bf I}_{1-\al}^{R}(|\mu|)(x)\right]^{\frac{1}{p-1}} + \left[ {\bf I}_{1-\al}^{R}(|\mu|)(y)\right]^{\frac{1}{p-1}}\right\} \nonumber\\
&\qquad +  \left(\frac{\de}{R}\right)^{\alpha} \left(\fint_{Q_{R}(x_0)} |\nabla u|^\kappa dz\right)^\frac{1}{\kappa}.
\end{align*}
This completes the proof because $\delta\leq \frac{1}{2}|x-y|$.
\end{proof}

\begin{remark}
	In Theorems \ref{0toi1}-\ref{gradholdTHM} and \ref{dinimax}-\ref{supdnvthm} we may take $\sigma_1=1$ in \eqref{Dini-VMO}, \eqref{dhcond} and \eqref{Dini-VMO-stronger}, provided we replace 
	$\om$ with a non-decreasing function $\tilde{\om}:[0,1]\rightarrow [0, \infty)$ such that $$\lim_{\rho\rightarrow 0} \tilde{\om}(\rho)=0,\quad  {\rm and} \quad 
	|A(x,\xi)-A(y,\xi)| \leq \tilde{\om}(|x-y|) |\xi|^{p-1}$$
	for all $x,y,\xi\in \RR^n$, $|x-y|\leq 1$. The reason is that in this case solutions to \eqref{eq1} are locally Lipschitz, and we can also take $\sigma_1=1$ in Lemma \ref{Dini-VMOlem}; see \cite[Section 8]{KM1}.
\end{remark}

\noindent \textbf{Acknowledgments:} 	Q.H.N.  is supported by the Academy of Mathematics and Systems Science, Chinese Academy of Sciences startup fund, and the National Natural Science Foundation of China (No. 12050410257 and No. 12288201) and  the National Key R$\&$D Program of China under grant 2021YFA1000800. N.C.P. is supported in part by Simons Foundation (award number 426071).

\end{document}